\documentclass{amsart}

\usepackage{amsmath}
\usepackage{amsthm}
\usepackage{amssymb}
\usepackage[alphabetic]{amsrefs}
\usepackage{setspace}
\usepackage{graphicx}

\newcommand{ \N } [0] { \mathbf{N} }
\newcommand{ \Z } [0] { \mathbf{Z} }
\newcommand{ \R } [0] { \mathbf{R} }
\newcommand{ \C } [0] { \mathbf{C} }

\DeclareMathOperator{\kernel}{ker}
\DeclareMathOperator{\opp}{opp}
\DeclareMathOperator{\aff}{aff}
\DeclareMathOperator{\myend}{End}
\DeclareMathOperator{\GL}{GL}
\DeclareMathOperator{\inducedrep}{ind}
\newcommand{ \cochargroup } [0] { X_{*} }

\newcommand{ \Aaff } [1] { \widetilde{A}_{#1} }
\newcommand{ \Caff } [1] { \widetilde{C}_{#1} }
\newcommand{ \Gaff } [1] { \widetilde{G}_{#1} }

\newcommand{ \heckefont } [1] {\mathcal{#1}}
\newcommand{ \vertexfont } [1] {\mathfrak{#1}}
\newcommand{ \alcovefont } [1] {\mathfrak{#1}}
\newcommand{ \apartmentfont } [1] {\mathcal{#1}}
\newcommand{ \chamberfont } [1] {\mathcal{#1}}
\newcommand{ \labelingfont } [1] {\mathbf{#1}}
\newcommand{ \labelingsymbol } [0] {\labelingfont{t}}
\newcommand{ \weylchambersymbol } [0] {\chamberfont{C}}
\newcommand{ \closedweylchambersymbol } [0] {\overline{\weylchambersymbol}}
\newcommand{ \vertexsymbol } [0] {\vertexfont{v}}
\newcommand{ \apartmentsymbol } [0] {\apartmentfont{A}}
\newcommand{ \enlargedapartment } [0] {\apartmentfont{V}}
\newcommand{ \basetag } [0] {\circ}
\newcommand{ \basereflections } [0] {\Delta_{\basetag}}
\newcommand{ \basevertex } [0] {\vertexsymbol_{\basetag}}
\newcommand{ \basealcove } [0] {\alcovefont{A}_{\basetag}}
\newcommand{ \basechamber } [0] {\weylchambersymbol_{\basetag}}
\newcommand{ \baselabeling } [0] {\labelingsymbol_{\basetag}}
\newcommand{ \closedbasealcove } [0] {\overline{\alcovefont{A}}_{\basetag}}
\newcommand{ \closedbasechamber } [0] {\closedweylchambersymbol_{\basetag}}
\newcommand{ \rootsystem } [0] {\Sigma}
\newcommand{ \scaledrootsystem } [0] {\Sigma}
\newcommand{ \corootlattice } [0] {Q}
\newcommand{ \translationsubgroup } [0] {\Lambda}
\newcommand{ \finiteweylgroup } [0] {W_{\basetag}}

\newcommand{ \defeq } [0] { \stackrel{\textup{def}}{=} }
\newcommand{ \suchthat } [0] { \hspace{5pt} \vert \hspace{5pt} }

\newcommand{ \QCGlabel } [1] {\textbf{QCG#1}}

\newtheorem*{defn}{Definition}
\newtheorem*{corollary}{Corollary}
\newtheorem*{maintheorem}{Main Theorem}
\newtheorem*{inductionlemma}{Induction Lemma}
\newtheorem{lemma}{Lemma}[subsection]
\newtheorem{prop}{Proposition}[subsection]
\newtheorem{remark}{Remark}[subsection]

\begin{document}

\title[Conjugacy of non-translations in affine Weyl groups]{Conjugacy Classes of Non-Translations in Affine Weyl Groups and Applications to Hecke Algebras}

%\date{\today}

\author{Sean Rostami}

\address{University of Wisconsin \\ Department of Mathematics \\ 480 Lincoln Dr. \\ Madison, WI 53706-1325 \\ United States}

\email{srostami@math.wisc.edu}

\email{sean.rostami@gmail.com}


\begin{abstract}
Let $ \widetilde{W} = \Lambda \rtimes W_{\circ} $ be an Iwahori-Weyl group of a connected reductive group $ G $ over a non-archimedean local field. The subgroup $ W_{\circ} $ is a finite Weyl group and the subgroup $ \Lambda $ is a finitely-generated abelian group (possibly containing torsion) which acts on a certain real affine space by translations. I prove that if $ w \in \widetilde{W} $ and $ w \notin \Lambda $ then one can apply to $ w $ a sequence of conjugations by simple reflections, each of which is length-preserving, resulting in an element $ w^{\prime} $ for which there exists a simple reflection $ s $ such that $ \ell ( s w^{\prime} ), \ell ( w^{\prime} s ) > \ell ( w^{\prime} ) $ and $ s w^{\prime} s \neq w^{\prime} $. Even for affine Weyl groups, a special case of Iwahori-Weyl groups and also an important subclass of Coxeter groups, this is a new fact about conjugacy classes. Further, there are implications for Iwahori-Hecke algebras $ \mathcal{H} $ of $ G $: one can use this fact to give dimension bounds on the ``length-filtration'' of the center $ Z ( \mathcal{H} ) $, which can in turn be used to prove that suitable linearly-independent subsets of $ Z ( \mathcal{H} ) $ are a basis.
\end{abstract}

\maketitle

\tableofcontents

\section{Introduction}

A \emph{Coxeter group} is a pair $ ( W, S ) $ consisting of a group $ W $ and a generating set $ S $ which is presented using the relations $ s^2 = 1 $ for all $ s \in S $ and relations of the form $ ( s t )^{m(s,t)} = 1 $ for some, but not necessarily all, pairs $ s, t \in S $. The most common examples of infinite Coxeter groups are \emph{affine Weyl groups}, which are groups generated by the reflections of an affine space across special collections of hyperplanes coming from root systems. The theory of Coxeter groups is both complicated and, especially in the case of affine Weyl groups, highly-developed. Affine Weyl groups are ubiquitous in the subject of smooth representations of algebraic groups over non-archimedean local fields due to their connection with Hecke algebras of reductive groups, and many questions about Hecke algebras can be reduced to questions about affine Weyl groups.

Let $ F $ be a non-archimedean local field and let $ G $ be a connected reductive affine algebraic $ F $-group. If $ J \subset G(F) $ is a compact-open subgroup and $ ( \rho, V ) $ is a smooth complex representation of $ J $ then the \emph{Hecke algebra} $ \heckefont{H} ( G ; J, \rho ) $ is the convolution algebra of all compactly-supported functions $ f : G(F) \rightarrow \myend_{\C} ( V ) $ satisfying $ f ( \jmath \cdot g \cdot \jmath^{\prime} ) = \rho ( \jmath ) \circ f ( g ) \circ \rho ( \jmath^{\prime} ) $ for all $ g \in G(F) $ and $ \jmath, \jmath^{\prime} \in J  $. Many of the simple subcategories in the \emph{Bernstein decomposition} of the category of smooth representations of $ G(F) $ are equivalent to the category of modules over a Hecke algebra of this form. The \emph{center} of such a Hecke algebra is important because it consists of the (functorial) $ G(F) $-linear endomorphisms of the representations in the subcategory. The case that $ J $ is an \emph{Iwahori subgroup} and $ \rho $ is the trivial $ 1 $-dimensional representation yields one particularly important Hecke algebra: the \emph{Iwahori-Hecke algebra}.

In general, an Iwahori subgroup is the group of $ \mathcal{O}_F $-points of a certain connected model $ \mathfrak{G} $ of $ G $ defined in general by Bruhat and Tits, although in nice cases like $ G = \GL_n $ there is a much more straightforward description: an Iwahori subgroup is the inverse image in $ G ( \mathcal{O}_F ) $ of a Borel subgroup in $ G ( \mathbf{k}_F ) $ under the reduction-mod-$ \pi $ map $ \mathcal{O}_F \rightarrow \mathbf{k}_F = \mathcal{O}_F / ( \pi_F ) $. It can be shown that any Iwahori-Hecke algebra $ \heckefont{H} $ has a presentation, called the \emph{Iwahori-Matsumoto presentation}, consisting of a basis of characteristic functions of the double-cosets $ \mathfrak{G} ( \mathcal{O}_F ) \backslash G ( F ) / \mathfrak{G} ( \mathcal{O}_F ) $ together with a certain pair of relations which depend on some numerical parameters coming from $ G $. It turns out that a group called the \emph{Iwahori-Weyl group} serves as a system of representatives for these double-cosets. The Iwahori-Weyl group is in general merely a semidirect extension of an affine Weyl group but its behavior is nonetheless extremely similar to that of a true Coxeter group. Taken together with the numerical parameters, the group-theoretic structure of this ``quasi-Coxeter group'' completely controls the ring-theoretic structure of $ \heckefont{H} $.

In this paper, I prove a group-theoretic property of conjugacy classes in Iwahori-Weyl groups, which is described precisely in the next subsection of this introduction. The class of all Iwahori-Weyl groups properly contains the class of all affine Weyl groups (since affine Weyl groups arise as the Iwahori-Weyl groups of semisimple and simply-connected $ G $) and this property is new even in this narrower context. Additionally, there are implications for Iwahori-Hecke algebras $ \heckefont{H} $. For example, this property can be used to show that suitable linearly-independent collections of functions in the center $ Z ( \heckefont{H} ) $ also \emph{span} the center. The prototypical example of such a collection is the \emph{Bernstein basis}, although the existence and precise definition of such a basis do not appear in the literature for Iwahori-Hecke algebras of completely general connected reductive groups. The article \cite{roro} fills this gap, which I explain in more detail next, and also serves as a sample application of the main theorem of this paper.

In unpublished work, concerning essentially only the case of split $ G $, Bernstein introduced a particularly important basis for $ Z ( \heckefont{H} ) $ whose elements can be effectively calculated (by computer, if desired) and which simultaneously have a straightforward representation-theoretic interpretation: the basis elements are indexed by $ \finiteweylgroup $-orbits of cocharacters valued in a certain maximal torus $ T \subset G $, and if $ \mathcal{O} = \{ \mu_1, \mu_2, \ldots, \mu_r \} $ is such an orbit then the corresponding basis element $ z_{\mathcal{O}} $ acts on the Iwahori-fixed subspace of the (normalized) induced representation $ \inducedrep ( \chi ) $ of an unramified character $ \chi : T ( F ) \rightarrow \C^{\times} $ by the scalar $ \chi ( \mu_1 ) + \chi ( \mu_2 ) + \cdots + \chi ( \mu_r ) $. This work was extended to the \emph{affine Hecke algebra} on any reduced root datum with any parameter system by Lusztig in \cite{lusztig}.

An affine Hecke algebra on a reduced root datum $ \Psi = ( X, \Phi, X^{\vee}, \Phi^{\vee} ) $ is constructed by using the \emph{extended affine Weyl group} $ W^{\prime} = X^{\vee} \rtimes \finiteweylgroup ( \Phi ) $ as a vector space basis, choosing parameters, and mimicking the Iwahori-Matsumoto presentation abstractly (see \S3.2 of \cite{lusztig} and \S7.1 of \cite{humphreys} for details). These abstractly defined affine Hecke algebras are useful in the study of reductive groups for the following reason: in some cases, for example if $ G $ is \emph{unramified}, any Iwahori-Hecke algebra of $ G $ is naturally isomorphic to an affine Hecke algebra for appropriate choice of root datum and parameters.

Unfortunately, many Iwahori-Weyl groups are \emph{not} the extended affine Weyl group of any root datum. A specific example of such a $ G $ is a ramified even-dimensional special orthogonal group (see \S1.16 in \cite{tits} for the definition of such a group): such a group has \emph{torsion} in the translation subgroup of its Iwahori-Weyl groups, which simply cannot happen in the extended affine Weyl group of a root datum because its translation subgroup $ X^{\vee} $ is free by definition. Consequently, many Iwahori-Hecke algebras are \emph{not} affine Hecke algebras and so are not within the scope of the Bernstein/Lusztig work.

In the article \cite{roro}, I extend the Bernstein/Lusztig results to Iwahori-Hecke algebras of \emph{all} connected reductive $ F $-groups. After establishing, analogous to \cite{lusztig}, a Bernstein presentation for an Iwahori-Hecke algebra $ \heckefont{H} $ of a general connected reductive $ F $-group, I determine, again following \cite{lusztig}, a Bernstein basis for $ Z ( \heckefont{H} ) $. To prove that these functions indeed span the center, I apply the main theorem of this paper as explained in \S\ref{Sapplicationtoheckealgebras} below.

\subsection{Statement of results}

\begin{center}
\emph{More precise definitions of everything here are given in \S\ref{Snotation} and \S\ref{Siwahoriweylgroups}.}
\end{center}

Let $ F $ be a non-archimedean local field and $ G $ a connected reductive affine algebraic $ F $-group. Fix a maximal $ F $-split torus $ A \subset G $ and let $ \widetilde{W} $ be the corresponding Iwahori-Weyl group, which acts on the vector space $ \enlargedapartment \defeq \cochargroup ( A ) \otimes_{\Z} \R $. It is known that $ \widetilde{W} $ contains as a normal subgroup the affine Weyl group $ W_{\aff} ( \scaledrootsystem ) $ of a reduced root system $ \scaledrootsystem $, and that if $ \translationsubgroup \subset \widetilde{W} $ is the subgroup of elements which act on $ \enlargedapartment $ by translations then $ \widetilde{W} $ splits as $ \widetilde{W} = \translationsubgroup \rtimes \finiteweylgroup ( \scaledrootsystem ) $ (here $ \finiteweylgroup $ denotes the finite Weyl group). Further, there are sections for $ \Omega \defeq \widetilde{W} / W_{\aff} ( \scaledrootsystem ) $, so that $ \widetilde{W} $ also splits as $ \widetilde{W} = W_{\aff} ( \scaledrootsystem ) \rtimes \Omega $. If $ \Delta_{\aff} $ is a Coxeter generating set for $ W_{\aff} ( \scaledrootsystem ) $ and $ \ell $ is the corresponding length function, then $ \ell $ extends to $ \widetilde{W} $ by inflation. Denote by $ \basereflections \subset \Delta_{\aff} $ a Coxeter generating set for $ \finiteweylgroup ( \scaledrootsystem ) $.

The main result of the paper is the following:
\begin{maintheorem}
Fix $ w \in \widetilde{W} $.

If $ w \notin \Lambda $ then there exists $ s \in \Delta_{\aff} $ and (if necessary) $ s_1, \ldots, s_n \in \Delta_{\aff} $ such that, setting $ w^{\prime} \defeq s_n \cdots s_1 w s_1 \cdots s_n $,
\begin{itemize}
\item $ \ell ( s_i \cdots s_1 w s_1 \cdots s_i ) = \ell ( w ) $ for all $ i $,

\item both $ \ell ( s w^{\prime} ) > \ell ( w^{\prime} ) $ and $ \ell ( w^{\prime} s ) > \ell ( w^{\prime} ) $, and

\item $ s w^{\prime} s \neq w^{\prime} $.
\end{itemize}
\end{maintheorem}
Note that the last two properties asserted by the Main Theorem can be unified into ``$ \ell ( s w^{\prime} s ) > \ell ( w^{\prime} ) $''.

\begin{remark}
A related result appears in the preprint \cite{HN}: that for any element $ w $ in the extended affine Weyl group of a (reduced) root datum there is a sequence of conjugations by simple reflections, each of which preserves or decreases length, resulting in an element $ w^{\prime} $ which is \emph{minimal} length in its conjugacy class. Note however that the \cite{HN} result is also true for \emph{finite} Weyl groups, originally proven by \cite{GP}, whereas the Main Theorem here is special to \emph{infinite} groups (length is bounded on a finite group!).
\end{remark}

Denote by $ \translationsubgroup / \finiteweylgroup $ the set of $ \finiteweylgroup ( \Sigma ) $-conjugacy classes in $ \translationsubgroup $ and recall that $ \ell $ is constant on any $ \mathcal{O} \in \translationsubgroup / \finiteweylgroup $. Denote by $ \Omega ( w ) $ the $ \Omega $-coordinate of any $ w \in \widetilde{W} $.

Let $ \heckefont{H} $ be an Iwahori-Hecke algebra for $ G $. By analyzing the equations that define the center $ Z ( \heckefont{H} ) $, the Main Theorem can be used to prove dimension bounds for a certain filtration/partition of $ Z ( \heckefont{H} ) $:
\begin{corollary}
If $ Z_{L,\tau} ( \heckefont{H} ) \subset Z ( \heckefont{H} ) $ is the $ \C $-subspace of functions supported only on those $ w $ for which $ \ell ( w ) \leq L $ and $ \Omega ( w ) = \tau $, and if $ N_{L,\tau} $ is the total number of $ \mathcal{O} \in \translationsubgroup / \finiteweylgroup $ such that $ \ell ( \mathcal{O} ) \leq L $ and $ \Omega ( t ) = \tau $ for all $ t \in \mathcal{O} $ then
\begin{equation*}
\dim_{\C} ( Z_{L,\tau} ( \heckefont{H} ) ) \leq N_{L,\tau}
\end{equation*}

It follows that if $ \{ z_{ \mathcal{O} } \}_{ \mathcal{O} \in \translationsubgroup / \finiteweylgroup } $ is a linearly-independent subset of $ Z ( \heckefont{H} ) $ such that $ z_{ \mathcal{O} } $ is supported only on those $ w $ for which $ \ell ( w ) \leq \ell ( \mathcal{O} ) $ and $ \Omega ( w ) $ is the same for all $ w $ supporting $ z_{ \mathcal{O} } $ then $ \{ z_{ \mathcal{O} } \}_{ \mathcal{O} \in \translationsubgroup / \finiteweylgroup } $ is a \emph{basis}.
\end{corollary}

\subsection{Outline of paper}

In \S\ref{Snotation}, I set some notation and define most of the objects that will be used throughout the paper. I use several non-standard but convenient notations, and almost all of them can be found here. Two exceptions are a ``quasi-Coxeter group'' and the Hecke algebra on such a group, which are treated in \S\ref{Siwahoriweylgroups} and \S\ref{SSdefofheckealgebra}, respectively.

In \S\ref{Siwahoriweylgroups}, I recall the notion of \emph{Iwahori-Weyl group} and several of its most important properties, mostly to make explicit the scope of the Main Theorem. Reductive groups do not appear after this section.

In \S\ref{Smarkedalcoves}, I define what is a \emph{marked alcove}. This is just a straightforward generalization of the notion of the \emph{type} of a face of an alcove. This extension is necessary to include Iwahori-Weyl groups, rather than just affine Weyl groups, in the scope of the paper.

In \S\ref{Sdiamonds}, I precisely define the \emph{Diamond Property} for an element of a Coxeter group. In short, the Diamond Property is the property asserted by the Main Theorem. I then define an equivalent property that refers only to pairs of alcoves and verify the equivalence of the two definitions.

In \S\ref{Slemmas}, I prove some simple geometric lemmas about Weyl chambers, alcoves, hyperplanes, etc. that will be used throughout \S\ref{Sproof}.

In \S\ref{Sproof}, I prove the Main Theorem, divided into three cases. In \S\ref{SSdominantcase}, I prove the ``dominant case'', Proposition \ref{PdominantnontranslationshaveDCP}, of the Main Theorem: \emph{if $ w \notin \Lambda $ and $ w $ sends the base alcove $ \basealcove $ into the dominant chamber $ \basechamber $ then $ w $ has the Diamond Property}. It is obvious from the hypothesis that $ \ell ( s w ) > \ell ( w ) $ for all $ s \in \basereflections $, and it is not hard to visualize why there also exists $ s \in \basereflections $ such that $ \ell ( w s ) > \ell ( w ) $. In \S\ref{SSantidominantcase}, I prove the ``anti-dominant case'', Proposition \ref{PantidominantnontranslationshaveCP}, of the Main Theorem: \emph{if $ w \notin \Lambda $ and $ w ( \basevertex ) \in \closedbasechamber^{\opp} $ then $ w $ has the Diamond Property}. This is the most difficult case and the general case can be reduced to this one (the complexity of the anti-dominant case is in some sense maximal while that of the dominant case is minimal--this can be quantified somewhat by noting that $ w ( \basealcove ) \subset \basechamber^{\opp} \Rightarrow \ell ( s w ) < \ell ( w ) $ for all $ s \in \basereflections $). The basic idea is that, by carefully inspecting the relative position and orientation of the alcoves $ \basealcove $ and $ w ( \basealcove ) $, one can perform an infinite sequence of conjugations by $ \Delta_{\aff} $ which do not \emph{decrease} length and which continually move the alcoves in ``different directions'', guaranteeing an eventual length-\emph{increasing} conjugation by $ \Delta_{\aff} $. In \S\ref{SSintermediatecase}, I finish proving the Main Theorem, showing that for an \emph{arbitrary} $ w \notin \Lambda $ one can repeatedly perform conjugations by $ \Delta_{\aff} $ which do not \emph{decrease} length and such that eventually the situation qualifies for the anti-dominant case.

In \S\ref{Sapplicationtoheckealgebras}, I use the Main Theorem to give dimension bounds, Proposition \ref{Pdimensionbounds}, for every subspace in the ``length filtration'' of $ Z ( \heckefont{H} ) $. The term ``length filtration'' is a slight abuse, since it is necessary to first filter $ Z ( \heckefont{H} ) $ by the lengths of its supporting elements and then partition each of those subspaces by the $ \Omega $-components of its supporting elements (if $ \Omega $ is infinite then the subspaces in the length filtration are \emph{infinite} dimensional, and without refining it further most Iwahori-Hecke algebras would be outside the scope of the paper).

In \S\ref{Sfigures}, I include several color pictures that illustrate the iterative arguments that occur in the proof of the Main Theorem. I use the affine Weyl group and apartment of the exceptional type $ \Gaff{2} $ because there are too many coincidences for extremely symmetric types like $ \Aaff{2} $ to correctly explain things, and because I can suppress hyperplane labels for $ \Gaff{2} $ since its labelings are, unlike $ \Aaff{2} $ and $ \Caff{2} $, unambiguous (the alcoves are $ 30^{\circ} $-$ 60^{\circ} $-$ 90^{\circ} $ triangles).

\subsection{Acknowledgements} I thank Thomas Haines for carefully reading earlier drafts of this paper and suggesting a very large number of changes to the exposition, and especially for noticing that the original proof occurring in \S\ref{SSintermediatecase} was needlessly labyrinthine. I also thank my postdoctoral mentor, Tonghai Yang, for bringing me to the wonderful University of Wisconsin at Madison. Finally, I thank the referee for suggesting many improvements to the introduction, and also for the care and speed with which the report was composed and returned.

\section{Notation and Setup} \label{Snotation}

The symbols $ \N $, $ \Z $, $ \R $, and $ \C $ refer to the natural numbers (including $ 0 $), the integers, the real numbers, and the complex numbers.

\subsection{Root systems and affine Weyl groups} \label{SSweylgroups}

Let $ \rootsystem $ be a \emph{reduced} and \emph{irreducible} root system. Let $ \finiteweylgroup $ be the finite Weyl group of $ \rootsystem $, let $ \basereflections $ be a simple system for $ \finiteweylgroup $.

Let $ \apartmentsymbol $ be the $ \R $-vector space spanned by the dual root system $ \rootsystem^{\vee} $. Let $ \langle -, - \rangle $ be the natural pairing $ \rootsystem^{\vee} \times \rootsystem \rightarrow \Z $ and $ \langle -, - \rangle_{\R} $ the extension to $ \apartmentsymbol \times \apartmentsymbol^{\vee} \rightarrow \R $. Let $ \rootsystem_{\aff} $ be the affine root system associated to $ \rootsystem $, i.e. the set $ \rootsystem + \Z $ of affine functionals on $ \apartmentsymbol $. The term \emph{hyperplane} always means the null-set of an element of $ \rootsystem_{\aff} $, not just an arbitrary codimension-$ 1 $ affine subspace. The term \emph{root hyperplane} means the null-set of an element of $ \rootsystem $.

Let $ W_{\aff} $ be the affine Weyl group of $ \rootsystem_{\aff} $ and $ \Delta_{\aff} $ the simple system for $ W_{\aff} $ extended from $ \basereflections $. Denote by $ s_{\aff} $ the single element of $ \Delta_{\aff} \setminus \basereflections $ and by $ H_{\aff} $ the hyperplane fixed pointwise by $ s_{\aff} $. Denote by $ \ell : W_{\aff} \rightarrow \N $ the usual length function relative to the generating set $ \Delta_{\aff} $. Let $ \corootlattice \subset W_{\aff} $ be the subgroup of translations by elements of $ \rootsystem^{\vee} $, so that $ W_{\aff} = \corootlattice \rtimes \finiteweylgroup $.

Choose a \emph{special} vertex $ \basevertex \in \apartmentsymbol $ and identify $ \finiteweylgroup $ to the finite Weyl group at $ \basevertex $. Let $ \basechamber $ be the Weyl chamber at $ \basevertex $ corresponding to $ \basereflections $ and let $ \basealcove \subset \basechamber $ be the alcove for which $ \basevertex \in \closedbasealcove $. If $ \weylchambersymbol $ is a Weyl chamber (at some arbitrary special vertex), denote by $ \weylchambersymbol^{\opp} $ the opposite chamber.

\subsection{Simplices and topology}

The term \emph{face} always means a codimension-$ 1 $ facet of an alcove (or Weyl chamber). The term \emph{wall} always means the unique hyperplane containing some face of some alcove (or Weyl chamber). If $ \alcovefont{A} $ and $ \alcovefont{B} $ are (distinct) adjacent alcoves then denote by $ \alcovefont{A} \vert \alcovefont{B} $ the wall separating them.

The term \emph{half-space} refers to one of the connected components of the complement in $ \apartmentsymbol $ of a hyperplane (in particular, half-spaces are open, and therefore also Weyl chambers and alcoves).

A subset $ \mathfrak{R} \subset \apartmentsymbol $ is called \emph{simplicial} iff there exists a set $ \mathbb{S} $ of alcoves in $ \apartmentsymbol $ such that $ \cup_{\alcovefont{A} \in \mathbb{S}} \alcovefont{A} \subset \mathfrak{R} \subset \cup_{\alcovefont{A} \in \mathbb{S}} \overline{\alcovefont{A}} $. For example, half-spaces and Weyl chambers are simplicial.

\subsection{Galleries and distances} \label{SSgalleries}

The \emph{length} of a gallery $ ( \alcovefont{B}_0, \alcovefont{B}_1, \ldots, \alcovefont{B}_n ) $ is defined to be $ n $ (for consistency with the length function $ \ell $). To say that a hyperplane $ H $ is an \emph{intermediate wall} of a (non-stuttering) gallery $ ( \alcovefont{B}_0, \alcovefont{B}_1, \ldots, \alcovefont{B}_n ) $ is the same as to say that there exists an index $ 0 \leq j < n $ such that $ H $ is the unique hyperplane containing the face shared by $ \alcovefont{B}_j $ and $ \alcovefont{B}_{j+1} $.

I sometimes use the fact that if $ \mathfrak{R} $ is simplicial and (topologically) convex then it is convex in the combinatorial sense, i.e. if $ \alcovefont{A}, \alcovefont{A}^{\prime} \subset \mathfrak{R} $ are alcoves and $ \mathcal{G} $ is a \emph{minimal} gallery from $ \alcovefont{A} $ to $ \alcovefont{A}^{\prime} $ then $ \alcovefont{A}^{\prime \prime} \subset \mathfrak{R} $ for \emph{all} $ \alcovefont{A}^{\prime \prime} \in \mathcal{G} $ (see Theorem 5.11.4 in \cite{BGW}). I sometimes refer to this property as \emph{simplicial convexity}. For example, half-spaces and Weyl chambers are simplicially-convex.

By \emph{infinite gallery} I mean an infinite sequence $ ( \alcovefont{B}_0, \alcovefont{B}_1, \ldots ) $ of alcoves such that for all $ i \geq 0 $ the alcoves $ \alcovefont{B}_i $ and $ \alcovefont{B}_{i+1} $ are adjacent. An infinite gallery is \emph{minimal} iff for \emph{all} $ 0 \leq i < j $ the finite sub-gallery $ ( \alcovefont{B}_i, \ldots, \alcovefont{B}_j ) $ is minimal in the usual sense.

Let $ d $ be the usual $ \N $-valued metric on the set of all alcoves: $ d ( \alcovefont{A}, \alcovefont{B} ) $ is defined to be the minimum length among all galleries from $ \alcovefont{A} $ and $ \alcovefont{B} $. Equivalently, $ d ( \alcovefont{A}, \alcovefont{B} ) $ is the total number of hyperplanes separating $ \alcovefont{A} $ from $ \alcovefont{B} $--see Theorem 5.1.4 in \cite{BGW}. I also use two extensions of this distance function $ d $:

If $ \mathfrak{R} \neq \emptyset $ is a \emph{simplicial} subset and $ \alcovefont{A} \subset \apartmentsymbol $ is an alcove then define
\begin{equation*}
d ( \alcovefont{A}, \mathfrak{R} ) \defeq \min ( d ( \alcovefont{A}, \alcovefont{B} ) \suchthat \text{all alcoves } \alcovefont{B} \subset \mathfrak{R} )
\end{equation*}

If $ \vertexfont{w} $ is a \emph{vertex}, and $ \mathfrak{R} $ is as before, then I define
\begin{equation*}
d ( \vertexfont{w}, \mathfrak{R} ) \defeq \min ( d ( \alcovefont{A}, \mathfrak{R} ) \suchthat \text{all alcoves $ \alcovefont{A} \subset \apartmentsymbol $ such that } \vertexfont{w} \in \overline{\alcovefont{A}} )
\end{equation*}

In applications, $ \mathfrak{R} $ will be either a single half-space or a Weyl chamber.

\section{Iwahori-Weyl Groups} \label{Siwahoriweylgroups}

In this section, I define the object which is the main focus of this paper: the \emph{Iwahori-Weyl group}.

\subsection{Definition and key properties} \label{SSdefofiwahoriweylgroup}

\begin{center}
\emph{In this subsection, I briefly explain what is an Iwahori-Weyl group and isolate its key properties. The purpose here is merely to explain the scope of the Main Theorem, so all proofs are omitted, although I give references whenever possible.}
\end{center}

Let $ F $ be a non-archimedean local field and let $ G $ be a connected reductive affine algebraic $ F $-group. Let $ A \subset G $ be a maximal $ F $-split torus and set $ M \defeq C_G ( A ) $, a minimal $ F $-Levi subgroup.

Certain group homomorphisms called \emph{Kottwitz homomorphisms} are very useful to understand the theory of parahoric subgroups, and in particular Iwahori subgroups. The Kottwitz homomorphism of a connected reductive affine algebraic $ F $-group $ H $ is a surjective group homomorphism $ \kappa_H : H(F) \twoheadrightarrow \Omega_H $, where $ \Omega_H $ is a finitely-generated abelian group whose precise definition is not relevant to this paper--see \S7 of \cite{kottwitz} for the definitions of $ \kappa_H $ and $ \Omega_H $ (the map $ \kappa_H $ occurring here is (7.7.1) in \cite{kottwitz}). The kernel of the Kottwitz homomorphism is denoted by $ H(F)_1 \defeq \kernel ( \kappa_H ) $.

One may define the \emph{Iwahori-Weyl group} of $ ( G, A ) $ to be the quotient $ \widetilde{W} \defeq N_G ( A ) ( F ) / M(F)_1 $. Note that this is seemingly different from the quotient occurring in Remark 9 of the Appendix to \cite{PRH}, but it can be shown that there is a tautological isomorphism between the two quotients. Proposition 8 combined with Remark 9 of the Appendix to \cite{PRH} shows if $ I \subset G(F) $ is an Iwahori subgroup then the double-cosets modulo $ I $ are naturally represented by $ \widetilde{W} $. The Iwahori-Weyl group $ \widetilde{W} $ acts on the vector space $ \enlargedapartment \defeq \cochargroup ( A ) \otimes_{\Z} \R $.

There are two extremely important ways to express the Iwahori-Weyl group $ \widetilde{W} $ as a semidirect product. By the work of Bruhat and Tits, it is known that there exists a reduced root system $ \scaledrootsystem $ such that the affine Weyl group $ W_{\aff} ( \scaledrootsystem ) $, in the sense of Ch VI \S2 no. 1 of \cite{bourbaki}, is a subgroup of $ \widetilde{W} $ (this root system is called an \emph{\'{e}chelonnage} in \S1.4 of \cite{BT1}; see \S4 of \cite{tits} for an extremely nice table listing $ \scaledrootsystem $ for every almost-simple group, and much more). Denoting by $ W_{\circ} ( \scaledrootsystem ) $ the finite Weyl group of $ \scaledrootsystem $, it can be shown that $ \widetilde{W} = \Omega_M \rtimes W_{\circ} ( \scaledrootsystem ) $ and that $ \widetilde{W} = W_{\aff} ( \scaledrootsystem ) \rtimes \Omega_G $. Further, the subgroup $ \Omega_M $ acts on $ \enlargedapartment $ by translations and the subgroup $ \Omega_G $ acts on $ \enlargedapartment $ by invertible affine transformations that stabilize any prescribed base alcove in $ \enlargedapartment $. For more details of all these semidirect products, consult \cite{PRH} and \cite{HRo}.

\begin{remark}
A few comments are necessary to emphasize the small but important difference between the notion of an Iwahori-Weyl group and the possibly more familiar notion of an extended affine Weyl group of a (reduced) root datum. First, it is possible that both $ \Omega_M $ and $ \Omega_G $ have torsion elements (in fact, it can be shown that the torsion of the former is contained in the torsion of the latter). Second, the elements of $ \Omega_M $ are not actually translations, but merely \emph{act} by translations on $ \enlargedapartment $. Third, it is possible that some non-identity elements in $ \Omega_M $ act by the identity on $ \enlargedapartment $.
\end{remark}

\subsection{Axiomatization}

Using the previous discussion \S\ref{SSdefofiwahoriweylgroup} as a guide, I now isolate the relevant properties of the Iwahori-Weyl group and present them axiomatically for clarity.

Let $ \mathcal{N} $ be the group of invertible affine transformations of $ \apartmentsymbol $ which normalize $ W_{\aff} $. Fix a finitely-generated abelian group $ \Omega $, a group homomorphism $ \psi : \Omega \rightarrow \mathcal{N} $, and act by $ \Omega $ on $ \apartmentsymbol $ via this $ \psi $.

\begin{defn}
The \emph{Quasi-Coxeter Group} $ \widetilde{W} $ extended from $ W_{\aff} $ by $ \Omega \stackrel{\psi}{\rightarrow} \mathcal{N} $ is the semidirect product $ W_{\aff} \rtimes \Omega $ and acts on $ \apartmentsymbol $ in the obvious way: $ ( w, \tau ) ( a ) \defeq w ( \tau ( a ) ) $ for all $ ( w, \tau ) \in \widetilde{W} $ and $ a \in \apartmentsymbol $. Denote by $ \Omega ( w ) $ the projection of $ w \in \widetilde{W} $ into $ \Omega $.
\end{defn}
Note that if $ w, w^{\prime} \in \widetilde{W} $ are conjugate then $ \Omega ( w ) = \Omega ( w^{\prime} ) $, since $ \Omega $ is abelian.

\begin{remark}
Strictly speaking, the space $ \apartmentsymbol $ on which the quasi-Coxeter group acts is only a proper subspace of the space $ \enlargedapartment $ on which the Iwahori-Weyl group acts when $ G $ is not semisimple. But due to the way that affine root hyperplanes in $ \enlargedapartment $ are defined, the details of which I omit in this paper, the difference is totally irrelevant from a group-theory perspective. The setup that I use is essentially the same as that used in Ch VI \S2 no. 3 of \cite{bourbaki}.
\end{remark}

Let $ \translationsubgroup \subset \widetilde{W} $ be the subgroup consisting of all elements that act by translations on $ \apartmentsymbol $, and note that $ \translationsubgroup $ is obviously normalized by $ \finiteweylgroup \subset W_{\aff} $. Extend the length function $ \ell : W_{\aff} \rightarrow \N $ to $ \widetilde{W} $ by inflation along the projection $ W_{\aff} \rtimes \Omega \rightarrow W_{\aff} $.

I impose the following hypotheses:
\begin{itemize}
\item \QCGlabel{1} Assume that $ \tau ( \basealcove ) = \basealcove $ for all $ \tau \in \Omega $.

\item \QCGlabel{2} Assume that $ \translationsubgroup $ is a semidirect complement, i.e. that $ \widetilde{W} = \translationsubgroup \rtimes \finiteweylgroup $.

\item \QCGlabel{3} Assume that $ \ell $ is constant on each $ \finiteweylgroup $-conjugacy class in $ \translationsubgroup $.

\item \QCGlabel{4} Assume that $ \translationsubgroup $ is finitely-generated and abelian.
\end{itemize}
Note that by choice of $ \mathcal{N} $, the action by $ \Omega $ on $ \apartmentsymbol $ permutes the set of hyperplanes in $ \apartmentsymbol $. Therefore, hypothesis \QCGlabel{1} is equivalent to the hypothesis that $ \tau ( \Delta_{\aff} ) = \Delta_{\aff} $ for all $ \tau \in \Omega $.

\section{Marked Alcoves} \label{Smarkedalcoves}

The definitions in this section, which are mostly just a variant on the notion of the \emph{type} of a face, will be used heavily in \S\ref{SSantidominantcase} and \S\ref{SSintermediatecase}.

\begin{defn}
A \emph{Labeling} of an alcove $ \alcovefont{A} \subset \apartmentsymbol $ is a bijection from $ \Delta_{\aff} $ to the set of walls of $ \alcovefont{A} $. A \emph{Marked Alcove} is a triple $ ( \alcovefont{A}, \vertexsymbol, \labelingsymbol ) $ such that $ \alcovefont{A} $ is an alcove, $ \vertexsymbol \in \overline{\alcovefont{A}} $ is a special vertex and $ \labelingsymbol $ is a labeling of $ \alcovefont{A} $. The \emph{Weyl Chamber} of a marked alcove $ ( \alcovefont{A}, \vertexsymbol, \labelingsymbol ) $ is the unique Weyl chamber at $ \vertexsymbol $ containing $ \alcovefont{A} $.
\end{defn}

Whenever the special vertex $ \vertexsymbol $ and labeling $ \labelingsymbol $ of a marked alcove $ ( \alcovefont{A}, \vertexsymbol, \labelingsymbol ) $ are understood and there is no danger of confusion, I abuse notation and refer to $ \alcovefont{A} $ as the marked alcove. Accordingly, if $ \alcovefont{A} $ represents a marked alcove then its special vertex is denoted by $ \vertexsymbol_{ \alcovefont{A} } $, its labeling by $ \labelingsymbol_{ \alcovefont{A} } $, its Weyl chamber by $ \weylchambersymbol_{ \alcovefont{A} } $, and the hyperplane $ \labelingsymbol_{\alcovefont{A}} ( s ) $ is called simply ``the wall of $ \alcovefont{A} $ labeled by $ s $''.

\begin{defn}
Two marked alcoves $ ( \alcovefont{A}, \vertexfont{v}, \labelingfont{t} ) $ and $ ( \alcovefont{B}, \vertexfont{w}, \labelingfont{s} ) $ are called \emph{Compatible} iff there exists $ w \in \widetilde{W} $ such that $ \alcovefont{B} = w ( \alcovefont{A} ) $, $ \vertexfont{w} = w ( \vertexfont{v} ) $ and $ \labelingfont{s} = w \circ \labelingfont{t} $. The marked alcoves are called \emph{NT-Compatible} iff $ w \notin \translationsubgroup $.
\end{defn}

Finally, the base alcove $ \basealcove $ is given the tautological labeling, and all other alcoves inherit (in general, multiple) labelings via the action of $ \widetilde{W} $ in the obvious way:
\begin{defn}
The \emph{Base Labeling} is the bijection $ \baselabeling $ from $ \Delta_{\aff} $ to the set of walls of the base alcove $ \basealcove $ defined by assigning to $ s $ the unique wall of $ \basealcove $ that is fixed pointwise by $ s $. The \emph{Base Marking} is the marked alcove $ ( \basealcove, \basevertex, \baselabeling ) $.

For each $ w \in \widetilde{W} $, the \emph{$ w $-Labeling} is defined to be the bijection $ \labelingsymbol_w \defeq w \circ \baselabeling $ from $ \Delta_{\aff} $ to the set of walls of the alcove $ w ( \basealcove ) $. The \emph{$ w $-Marked Alcove} is by definition the triple $ ( w ( \basealcove ), w ( \basevertex ), \labelingsymbol_w ) $.
\end{defn}

As before, I sometimes abuse notation by using $ w ( \basealcove ) $ to refer to the $ w $-marked alcove. Note that $ w ( \basealcove ) $ is compatible with $ \basealcove $ and it is NT-compatible with $ \basealcove $ if and only if $ w \notin \translationsubgroup $.

\begin{remark}
When $ \Omega = \{ 1 \} $, alcoves are in bijection with $ w $-marked alcoves (due to simple-transitivity of affine Weyl groups on alcoves) and a labeling is essentially just the assignment to every face of every alcove its \emph{type} in the usual way.
\end{remark}

The following operation will be used frequently in the limiting/inductive arguments of \S\ref{SSantidominantcase} and \S\ref{SSintermediatecase}:
\begin{defn}
For any marked alcove $ ( \alcovefont{A}, \vertexfont{v}, \labelingfont{t} ) $ and any $ s \in \Delta_{\aff} $, the marked alcove $ ( \alcovefont{A}, \vertexfont{v}, \labelingfont{t} )^s $ is by definition the triple $ ( s_H ( \alcovefont{A} ), s_H ( \vertexfont{v} ), s_H \circ \labelingfont{t} ) $, where $ H \defeq \labelingfont{t} ( s ) $ is the wall of $ \alcovefont{A} $ labeled by $ s $.
\end{defn}

Note that if two marked alcoves $ \alcovefont{A} $ and $ \alcovefont{B} $ are NT-compatible then $ \alcovefont{A}^s $ and $ \alcovefont{B}^s $ are also NT-compatible for all $ s \in \Delta_{\aff} $.

I frequently use the fact that applying a sequence of various $ * \mapsto *^s $ operations to a single marked alcove results in a sequence of marked alcoves whose (un-marked) alcoves form a \emph{gallery}. Conversely, if $ \alcovefont{B}_0 $ is a marked alcove and $ ( \alcovefont{B}_0, \ldots, \alcovefont{B}_n ) $ is a gallery with no repeated alcoves, then each $ \alcovefont{B}_i $ becomes a marked alcove in a unique way, by iteratively applying $ * \mapsto *^s $ operations across each intermediate wall of the gallery. In this situation, the ``label'' of $ \alcovefont{B}_i \vert \alcovefont{B}_{i+1} $ is understood to refer to the element of $ \Delta_{\aff} $ corresponding to the wall $ \alcovefont{B}_i \vert \alcovefont{B}_{i+1} $ relative to the labeling of $ \alcovefont{B}_i $ (or $ \alcovefont{B}_{i+1} $) inherited from $ \alcovefont{B}_0 $.

\section{Diamond Properties} \label{Sdiamonds}

\subsection{Lateral-conjugacy and the diamond property in the group}

\begin{defn}
$ w, w^{\prime} \in \widetilde{W} $ are \emph{Laterally-Conjugate} iff there exist $ s_1, \ldots, s_n \in \Delta_{\aff} $ such that $ w^{\prime} = s_n \cdots s_1 w s_1 \cdots s_n $ and $ \ell ( s_i \cdots s_1 w s_1 \cdots s_i ) = \ell ( w ) $ for all $ i $.
\end{defn}

Any $ w \in \widetilde{W} $ is always considered to be laterally-conjugate to itself.

\begin{defn}
$ w \in \widetilde{W} $ has the \emph{Direct Diamond Property} iff there exists $ s \in \Delta_{\aff} $ such that
\begin{itemize}
\item $ s w s \neq w $,

\item $ \ell ( s w ) > \ell ( w ) $, and

\item $ \ell ( w s ) > \ell ( w ) $.
\end{itemize}
\end{defn}

By using the well-known Lemma \ref{Lhumphreys}, these three properties \emph{could} be replaced by the single property ``$ \ell ( s w s ) > \ell ( w ) $'', but this formulation is not as convenient for me.

\begin{defn}
$ w \in \widetilde{W} $ has the \emph{Diamond Property} iff it is laterally-conjugate to an element with the Direct Diamond Property.
\end{defn}

I frequently use the following geometric characterization of length:
\begin{lemma} \label{Lgeometriclength}
Let $ w \in \widetilde{W} $ and $ s \in \Delta_{\aff} $ be arbitrary. If $ H \defeq \baselabeling ( s ) $ and $ K \defeq \labelingsymbol_w ( s ) $ then
\begin{itemize}
\item $ \ell ( s w ) > \ell ( w ) $ if and only if $ \basealcove $ and $ w ( \basealcove ) $ are on the \emph{same} side of $ H $,

\item $ \ell ( w s ) > \ell ( w ) $ if and only if $ \basealcove $ and $ w ( \basealcove ) $ are on the \emph{same} side of $ K $, and

\item $ \ell ( w ) = d ( \basealcove, w ( \basealcove ) ) $.
\end{itemize}
\end{lemma}

\begin{proof}
When $ \Omega = \{ 1 \} $ this is all well-known: see Proposition (c) in \S4.4 and Theorem (b) in \S4.5 of \cite{humphreys}. The more general statement is immediate by definition of the labeling $ \labelingsymbol_w $ because $ \ell $ factors through $ W_{\aff} $ and $ \Omega $ stabilizes $ \basealcove $.
\end{proof}

\subsection{Lateral-conjugacy and the diamond property in the apartment}

Here are the gallery-theoretic versions of the above 3 definitions:
\begin{defn}
An ordered pair $ ( \alcovefont{A}, \alcovefont{B} ) $ of marked alcoves is \emph{Laterally-Conjugate} to another pair $ ( \alcovefont{A}^{\prime}, \alcovefont{B}^{\prime} ) $ iff there exists a gallery $ ( \alcovefont{A}_0, \alcovefont{A}_1, \ldots, \alcovefont{A}_n ) $ from $ \alcovefont{A} $ to $ \alcovefont{A}^{\prime} $ and a gallery $ ( \alcovefont{B}_0, \alcovefont{B}_1, \ldots, \alcovefont{B}_n ) $ from $ \alcovefont{B} $ to $ \alcovefont{B}^{\prime} $ such that
\begin{itemize}
\item $ d ( \alcovefont{A}_i, \alcovefont{B}_i ) = d ( \alcovefont{A}, \alcovefont{B} ) $, and

\item both $ \alcovefont{A}_i \vert \alcovefont{A}_{i+1} $ and $ \alcovefont{B}_i \vert \alcovefont{B}_{i+1} $ have the \emph{same} label.
\end{itemize}
for all $ i $.
\end{defn}

Note the symmetry in the definition of lateral-conjugacy: $ ( \alcovefont{A}, \alcovefont{B} ) $ is laterally-conjugate to $ ( \alcovefont{A}^{\prime}, \alcovefont{B}^{\prime} ) $ if and only if $ ( \alcovefont{B}, \alcovefont{A} ) $ is laterally-conjugate to $ ( \alcovefont{B}^{\prime}, \alcovefont{A}^{\prime} ) $.

\begin{defn}
A pair $ \{ \alcovefont{A}, \alcovefont{B} \} $ of marked alcoves has the \emph{Direct Diamond Property} iff there exists $ s \in \Delta_{\aff} $ such that
\begin{itemize}
\item $ \mathbf{t}_{\alcovefont{A}} ( s ) \neq \mathbf{t}_{\alcovefont{B}} ( s ) $,

\item both alcoves are on the same side of $ \mathbf{t}_{\alcovefont{A}} ( s ) $, and

\item both alcoves are also on the same side of $ \mathbf{t}_{\alcovefont{B}} ( s ) $.
\end{itemize}
\end{defn}

Note that, due to the symmetry in the definition, the Direct Diamond Property refers to \emph{unordered} pairs of alcoves.

\begin{defn}
A pair $ \{ \alcovefont{A}, \alcovefont{B} \} $ of marked alcoves has the \emph{Diamond Property} iff it is laterally-conjugate to a pair with the direct diamond property.
\end{defn}

\subsection{Equivalence}

It is easy to show using the previous lemma that the two notions of Diamond Property coincide:
\begin{lemma} \label{Lequivalentdiamonds}
The element $ w \in \widetilde{W} $ has the Diamond Property if and only if the pair $ \{ \basealcove, w ( \basealcove ) \} $ of marked alcoves has the Diamond Property.
\end{lemma}

\begin{proof}
By Lemma \ref{Lgeometriclength}, $ \ell ( s w s ) = d ( \basealcove, s w s ( \basealcove ) ) $. Since $ d $ is invariant under the diagonal action of $ W_{\aff} $, $ d ( \basealcove, s w s ( \basealcove ) ) = d ( s ( \basealcove ), w s ( \basealcove ) ) $. By definition of the labelings $ \baselabeling $ and $ \labelingsymbol_w $, $ d ( s ( \basealcove ), w s ( \basealcove ) ) = d ( \basealcove^s, w ( \basealcove )^s ) $. Altogether, $ \ell ( s w s ) = d ( \basealcove^s, w ( \basealcove )^s ) $. Since the operation $ * \mapsto *^s $ always creates galleries, this shows that the two notions of ``lateral-conjugacy'' are equivalent. Since $ s_H w s_H = w $ if and only if $ w ( H ) = H $, it follows that the condition $ s w s \neq w $ is equivalent to the condition $ \labelingsymbol_w ( s ) \neq \baselabeling ( s ) $. Finally, the fact that the remaining two statements in both Direct Diamond Properties are equivalent follows directly from Lemma \ref{Lgeometriclength}.
\end{proof}

Although I will have no direct use for this in the remainder of the paper, note that the statement of Lemma \ref{Lequivalentdiamonds} could be made much more specific: if $ w $ is laterally-conjugate to $ w^{\prime} $ via the sequence $ s_1, \ldots, s_r $ then $ ( \basealcove, w ( \basealcove ) ) $ is laterally-conjugate to $ ( s_1 \cdots s_r ( \basealcove ), w s_1 \cdots s_r ( \basealcove ) ) $ via a gallery whose intermediate walls are labeled (in tandem) by the same sequence $ s_1, \ldots, s_r $, etc.

\section{Basic Lemmas} \label{Slemmas}

\begin{lemma} \label{Laffinehyperplanesarenotbipolar}
Let $ H $ be a hyperplane in $ \apartmentsymbol $.

If \emph{both} $  \closedbasechamber \cap H \neq \emptyset $ \emph{and} $ \closedbasechamber^{\opp} \cap H \neq \emptyset $ then in fact both $ \closedbasechamber \cap H \subset \partial \basechamber $ and $ \closedbasechamber^{\opp} \cap H \subset \partial \basechamber^{\opp} $.
\end{lemma}

\begin{remark}
I usually apply Lemma \ref{Laffinehyperplanesarenotbipolar} in the following way: if $ H \cap \basechamber \neq \emptyset $ then $ H \cap \closedbasechamber^{\opp} = \emptyset $.
\end{remark}

\begin{proof}
Suppose $ x \in H \cap \closedbasechamber $ and $ y \in H \cap \closedbasechamber^{\opp} $. Since $ H $ is the null-set of an affine root, there exists $ \beta \in \rootsystem $ such that $ x - y \in H_{\beta} $. On the other hand, $ x \in \closedbasechamber $ and $ y \in \closedbasechamber^{\opp} $ implies $ x - y \in \closedbasechamber $. It is easy to check from the definitions of $ \basechamber $ and $ \basechamber^{\opp} $ that if $ x \in \basechamber $ or $ y \in \basechamber^{\opp} $ (or both) then necessarily $ x - y \in \basechamber $. But $ H_{\beta} \cap \basechamber = \emptyset $ since a \emph{root} hyperplane can never intersect a Weyl chamber at $ \basevertex $, so this is impossible and therefore both $ x \notin \basechamber $ and $ y \notin \basechamber^{\opp} $.
\end{proof}

\begin{lemma} \label{Lnegativechambercontainment}
Let $ \alcovefont{A} \subset \basechamber $ be an alcove and $ H $ a wall of $ \alcovefont{A} $ that is \emph{not} a wall of the Weyl chamber $ \basechamber $. Let $ \alcovefont{B} $ be any alcove and $ \vertexsymbol_{ \alcovefont{B} } \in \overline{\alcovefont{B}} $ a vertex.

If both $ \alcovefont{A} $ and $ \basealcove $ are on the same side of $ H $ and $ \vertexsymbol_{ \alcovefont{B} } \in \closedbasechamber^{\opp} $ then $ \alcovefont{B} $ is also on the same side of $ H $ as $ \alcovefont{A} $.
\end{lemma}

\begin{proof}
Let $ \alcovefont{f} \subset \overline{\alcovefont{A}} $ be the face supported by $ H $. If it were true that $ \alcovefont{f} \subset \partial \basechamber $ then necessarily $ H $ would be a wall of $ \basechamber $. This is prohibited by hypothesis on $ H $, so $ H \cap \basechamber \neq \emptyset $. Suppose for contradiction that $ H $ separated $ \alcovefont{B} $ from $ \alcovefont{A} $. By hypothesis on $ \alcovefont{A} $, $ H $ must separate $ \alcovefont{B} $ from $ \basealcove $. By hypothesis on $ \vertexsymbol_{ \alcovefont{B} } $, the set $ \alcovefont{B} \cup \{ \vertexsymbol_{ \alcovefont{B} } \} \cup \basechamber^{\opp} \cup \{ \basevertex \} \cup \basealcove $ is path-connected and obviously $ H \cap ( \alcovefont{B} \cup \basealcove ) = \emptyset $, so necessarily $ H \cap \closedbasechamber^{\opp} \neq \emptyset $. But this contradicts Lemma \ref{Laffinehyperplanesarenotbipolar} since $ H \cap \basechamber \neq \emptyset $ is known already.
\end{proof}

\begin{lemma} \label{Ldoublechamberlemma}
Let $ \weylchambersymbol $ be a Weyl chamber at some (arbitrary) special vertex $ \vertexsymbol $.

If $ \vertexsymbol \in \closedbasechamber^{\opp} $ then either $ \weylchambersymbol \cap \basechamber = \emptyset $ or $ \weylchambersymbol \supset \basechamber $.
\end{lemma}

\begin{proof}
Suppose for contradiction that both $ \weylchambersymbol \cap \basechamber \neq \emptyset $ and $ \weylchambersymbol \not\supset \basechamber $. Choose an alcove $ \alcovefont{A} \subset \basechamber \setminus \weylchambersymbol $ and an alcove $ \alcovefont{A}^{\prime} \subset \basechamber \cap \weylchambersymbol $. Let $ ( \alcovefont{A}_1, \ldots, \alcovefont{A}_n ) $ be a minimal gallery from $ \alcovefont{A} $ to $ \alcovefont{A}^{\prime} $. By choice of $ \alcovefont{A}, \alcovefont{A}^{\prime} $, there exists a wall $ H $ of $ \weylchambersymbol $ separating $ \alcovefont{A} $ from $ \alcovefont{A}^{\prime} $. Such an $ H $ must be an intermediate wall of the gallery (see Lemma 5.1.5 of \cite{BGW}), say $ H = \alcovefont{A}_j \vert \alcovefont{A}_{j+1} $. Since the gallery is \emph{minimal} and both $ \alcovefont{A}, \alcovefont{A}^{\prime} \subset \basechamber $, simplicial convexity of $ \basechamber $ forces $ \alcovefont{A}_i \subset \basechamber $ for all $ i $. This means that $ \basechamber \cap H \neq \emptyset $, (for example, if $ \alcovefont{f} $ is the common face of $ \alcovefont{A}_j $ and $ \alcovefont{A}_{j+1} $ then $ \alcovefont{f} \subset \basechamber \cap H $). But by hypothesis $ \vertexsymbol \in \closedbasechamber^{\opp} $ and obviously $ \vertexsymbol \in H $, so $ \closedbasechamber^{\opp} \cap H \neq \emptyset $ also. This contradicts Lemma \ref{Laffinehyperplanesarenotbipolar}.
\end{proof}

\begin{lemma} \label{Lgallerytoinfinity}
Let $ \weylchambersymbol $ be a Weyl chamber at some (arbitrary) special vertex.

If $ \basechamber \cap \weylchambersymbol = \emptyset $ then there exists a minimal infinite gallery $ ( \alcovefont{A}_0 = \basealcove, \alcovefont{A}_1, \alcovefont{A}_2, \ldots ) $ within $ \basechamber $ such that
\begin{equation*}
\lim_{i \rightarrow \infty} d ( \alcovefont{A}_i, \weylchambersymbol ) = \infty.
\end{equation*}
(see \S\ref{SSgalleries} for the notions of distance $ d ( * , \weylchambersymbol ) $ and infinite gallery)
\end{lemma}

\begin{remark}
It is not always possible to have a sequence of alcoves for which the sequence of distances is \emph{monotone} increasing.
\end{remark}

\begin{proof}
Choose some translation $ t \in \corootlattice $ such that $ t ( \basevertex ) \in \basechamber $ (e.g. translation by $ 2 \rho^{\vee} $ where $ \rho^{\vee} \defeq \omega_1 + \cdots + \omega_r $ and $ \omega_i $ are the fundamental coweights). I claim that the required gallery can be constructed by iterating $ t $.

Let $ t = s_0 s_1 \cdots s_{k-1} $ ($ s_i \in \Delta_{\aff} $) be a reduced expression and let $ [i] $ be the remainder of $ i $ mod $ k $. Give $ \alcovefont{A}_0 \defeq \basealcove $ the base marking as usual and define a sequence of marked alcoves inductively by $ \alcovefont{A}_{i+1} \defeq \alcovefont{A}_i^{s_{[i]}} $ (so the finite subgallery $ ( \alcovefont{A}_0, \alcovefont{A}_1, \ldots, \alcovefont{A}_k ) $ is just the usual gallery associated to the word $ s_0 s_1 \cdots s_{k-1} $). By construction, this sequence is an \emph{infinite gallery} in $ \basechamber $. Since $ \ell ( t^N ) = N \ell ( t ) $, a general property of dominant translations in an affine Weyl group, it follows that the infinite gallery $ ( \alcovefont{A}_0, \alcovefont{A}_1, \alcovefont{A}_2, \ldots ) $ is \emph{minimal}.

I first show that the infinite subsequence $ \alcovefont{A}_{k N} = t^N ( \basealcove ) $, $ N = 0, 1, 2, \ldots $, diverges from $ \weylchambersymbol $, and then I use the triangle-inequality to prove the full limit property.

For any alcove $ \alcovefont{A} \subset \apartmentsymbol $ and radius $ R \in \R $, let $ \mathbb{B} ( \alcovefont{A}, R ) $ be the set of alcoves $ \alcovefont{B} \subset \apartmentsymbol $ such that $ d ( \alcovefont{A}, \alcovefont{B} ) \leq R $. It is clear from the ``cone'' property of Weyl chambers and the boundedness of alcoves that for any $ R \in \R $, there exists $ n_R \in \N $ such that $ \mathbb{B} ( t^N ( \basealcove ), R ) \subset \basechamber $ for all $ N \geq n_R $. It is also clear that $ d ( t^N ( \basealcove ), \weylchambersymbol ) > R $ for all $ N \geq n_R $ because otherwise there would be some alcove $ \alcovefont{B} \subset \weylchambersymbol $ such that $ d ( t^N ( \basealcove ), \alcovefont{B} ) \leq R $, but this would imply $ \alcovefont{B} \subset \mathbb{B} ( t^N ( \basealcove ), R ) \subset \basechamber $, which contradicts the hypothesis $ \basechamber \cap \weylchambersymbol = \emptyset $. This establishes the claim for the subsequence.

Let radius $ R > 0 $ be arbitrary. Fix $ n \defeq k \cdot n_{R+k} $ (recall $ k = \ell ( t ) $). For any $ N \in \N $, let $ \lfloor N \rfloor $ be the \emph{largest} $ m \in \N $ such that $ k m \leq N $. Observe that if $ N \geq n $ then $ \lfloor N \rfloor \geq n_{R+k} $. Altogether, if $ \alcovefont{B} \subset \weylchambersymbol $ is an \emph{arbitrary} alcove and $ N \geq n $ then
\begin{equation*}
R + k < d ( t^{ \lfloor N \rfloor } ( \basealcove ), \alcovefont{B} ) \leq d ( t^{ \lfloor N \rfloor } ( \basealcove ), \alcovefont{A}_N ) + d ( \alcovefont{A}_N, \alcovefont{B} )
\end{equation*}
Since $ d ( t^{ \lfloor N \rfloor } ( \basealcove ), \alcovefont{A}_N ) \leq \ell ( t ) = k $ by definition of $ \lfloor N \rfloor $, it follows that $ d ( \alcovefont{A}_N, \alcovefont{B} ) > R $.
\end{proof}

\section{Proof of Main Theorem} \label{Sproof}

\subsection{Case: the dominant chamber} \label{SSdominantcase}

\begin{prop} \label{PdominantnontranslationshaveDCP}
Fix $ w \in \widetilde{W} $.

If $ w \notin \translationsubgroup $ and $ w ( \basealcove ) \subset \basechamber $ then $ w $ has the Direct Diamond Property realized by some $ s \in \basereflections $.
\end{prop}

\begin{proof}
Let $ u \in \finiteweylgroup $ and $ t \in \translationsubgroup $ be such that $ w = t \circ u $. By hypothesis, $ u \neq 1 $. Let $ H^{\prime} $ be a wall of $ u ( \basechamber ) $ which separates $ u ( \basealcove ) $ from $ \basealcove $, and note that $ \basevertex \in H^{\prime} $. Let $ H $ be the wall of $ \basealcove $ such that $ u ( H ) = H^{\prime} $ and let $ s \in \basereflections $ be the element fixing $ H $ pointwise (in other words, $ H^{\prime} = \mathbf{t}_u ( s ) $). I claim that $ s $ realizes the Direct Diamond Property for $ w $.

I first claim that $ t ( H^{\prime} ) \neq H^{\prime} $. Suppose for contradiction that $ t ( H^{\prime} ) = H^{\prime} $. Because $ t $ is a \emph{translation}, $ u ( \basealcove ) $ and $ t ( u ( \basealcove ) ) $ are on the \emph{same} side of $ t ( H^{\prime} ) = H^{\prime} $. On the other hand, $ H^{\prime} $ separates $ \basealcove $ from $ u ( \basealcove ) $ by choice. Together, $ H^{\prime} $ separates $ \basealcove $ from $ t ( u ( \basealcove ) ) = w ( \basealcove ) $. But $ w ( \basealcove ) \subset \basechamber $ by hypothesis, so it is impossible for the \emph{root} hyperplane $ H^{\prime} $ to separate $ \basealcove $ from $ w ( \basealcove ) $.

By hypothesis that $ w ( \basealcove ) \subset \basechamber $, it is automatic that $ \ell ( s w ) > \ell ( w ) $. To show that $ \ell ( w s ) > \ell ( w ) $, it suffices by Lemma \ref{Lgeometriclength} to show that both alcoves $ \basealcove $ and $ w ( \basealcove ) $ are on the same side of $ w ( H ) $. Let $ \alpha \in \rootsystem $ be the \emph{positive} root whose null-set is $ H^{\prime} $. By choice of $ H^{\prime} $, $ \langle \alpha, x \rangle_{\R} < 0 $ for all $ x \in u ( \basealcove ) $. Since $ u ( \basevertex ) = \basevertex $ and $ w ( \basevertex ) \in \closedbasechamber $, it must be true that $ t ( \basevertex ) \in \closedbasechamber $. Since $ t $ is a \emph{translation}, this implies that there exists $ n \in \N $ such that $ t ( H^{\prime} ) = w ( H ) $ is the null-set of $ \alpha - n $ and $ t ( u ( \basealcove ) ) = w ( \basealcove ) $ consists of points $ x \in \apartmentsymbol $ such that $ \langle \alpha, x \rangle_{\R} < n $. Since $ t ( H^{\prime} ) \neq H^{\prime} $, it must be true that $ n \geq 1 $. But $ 0 < \langle \alpha, x \rangle_{\R} < 1 \leq n $ for all $ x \in \basealcove $ so $ \basealcove $ and $ w ( \basealcove ) $ are on the same side of $ w ( H ) $, as desired.

I now show that $ s w s \neq w $. Suppose for contradiction that $ s w s = w $. Then $ w ( H ) = H $, and since $ u ( \basevertex ) = \basevertex $, it follows that $ t ( \basevertex ) \in H $. Since $ t $ is a \emph{translation} and $ \basevertex \in H $, $ t ( H ) = H $. Combining with $ w ( H ) = H $ implies $ u ( H ) = H $ also. But $ H^{\prime} \defeq u ( H ) $ so this contradicts $ t ( H^{\prime} ) \neq H^{\prime} $.
\end{proof}

\begin{remark} \label{Rindependenceofchamber}
Note that it is not important which of the two alcoves is considered the ``base'' alcove, nor is it important which chamber of the base alcove is considered ``dominant''. In other words, if $ \alcovefont{A} $ and $ \alcovefont{B} $ are NT-compatible marked alcoves and $ \alcovefont{B} \subset \weylchambersymbol_{\alcovefont{A}} $ then $ \{ \alcovefont{A}, \alcovefont{B} \} $ has the Direct Diamond Property.
\end{remark}

\begin{remark}
It is plausible that one might be able to prove the Diamond Property for general $ w \notin \translationsubgroup $ by proving that there always exists a lateral conjugate $ w^{\prime} $ of $ w $ such that one of $ w^{\prime} ( \basealcove ) $ or $ \basealcove $ is contained in some Weyl chamber of the other. This latter statement is \emph{false}. See Figure \ref{example-g2-annoyingorbit} for an example in the case of the exceptional affine Weyl group $ \Gaff{2} $.
\end{remark}

\subsection{Case: the anti-dominant chamber} \label{SSantidominantcase}

\begin{defn}
Let $ ( \alcovefont{B}_0, \ldots, \alcovefont{B}_n ) $ be a gallery, $ \alcovefont{A} $ an alcove, and $ H $ a wall of $ \alcovefont{A} $.

The triple $ ( ( \alcovefont{B}_0, \ldots, \alcovefont{B}_n ), \alcovefont{A}, H ) $ is an \emph{Umbrella} iff
\begin{enumerate}
\item \label{UMBsamehalfspace} \emph{all} alcoves $ \alcovefont{B}_i $ are on the same side of $ H $ as $ \alcovefont{A} $, and
\item \label{UMBextendstominimal} $ ( \alcovefont{B}_0, \ldots, \alcovefont{B}_n ) $ can be extended to a \emph{minimal} gallery from $ \alcovefont{B}_0 $ to $ \alcovefont{A} $.
\end{enumerate}
\end{defn}

Observe that to say $ ( \alcovefont{B}_0, \ldots, \alcovefont{B}_n ) $ can be extended to a minimal gallery from $ \alcovefont{B}_0 $ to $ \alcovefont{A} $ is the same as to say both that $ ( \alcovefont{B}_0, \ldots, \alcovefont{B}_n ) $ is a minimal gallery itself and that each intermediate wall $ \alcovefont{B}_i \vert \alcovefont{B}_{i+1} $ ($ 0 \leq i < n $) separates $ \alcovefont{B}_i $ from $ \alcovefont{A} $ (I use this observation in the proof of Induction Lemma).

\begin{inductionlemma}
Let $ ( \alcovefont{B}_0, \ldots, \alcovefont{B}_n ) $ be a gallery and $ \vertexsymbol_{\alcovefont{B}_0} \in \overline{\alcovefont{B}}_0 $ a special vertex. Let $ \alcovefont{A}, \alcovefont{A}^{\prime} \subset \basechamber $ be (distinct) adjacent alcoves, separated by a wall $ H $. Let $ H^{\prime} $ be a wall of $ \alcovefont{A}^{\prime} $. Assume that
\begin{enumerate}
\item \label{ENUMinductionhypoth} $ ( ( \alcovefont{B}_0, \ldots, \alcovefont{B}_{n-1} ), \alcovefont{A}, H ) $ is an Umbrella,

\item \label{ENUMvertexinchamber} $ \vertexsymbol_{\alcovefont{B}_0} \in \closedbasechamber^{\opp} $,

\item \label{ENUMawayfrombase} the base alcove $ \basealcove $ is on the \emph{same} side of $ H^{\prime} $ as $ \alcovefont{A}^{\prime} $,

\item \label{ENUMnotbasewall} $ H^{\prime} $ is \emph{not} a wall of the Weyl chamber $ \basechamber $, and

\item \label{ENUMcontradictionhypoth} the wall $ \alcovefont{B}_{n-1} \vert \alcovefont{B}_n $ \emph{separates} $ \alcovefont{B}_{n-1} $ from $ \alcovefont{A} $.
\end{enumerate}

Then $ ( ( \alcovefont{B}_0, \ldots, \alcovefont{B}_n ), \alcovefont{A}^{\prime}, H^{\prime} ) $ is an Umbrella.
\end{inductionlemma}

\begin{proof}
By hypotheses (\ref{ENUMvertexinchamber}), (\ref{ENUMawayfrombase}), and (\ref{ENUMnotbasewall}), Lemma \ref{Lnegativechambercontainment} implies that $ \alcovefont{B}_0 $ is contained on the same side of $ H^{\prime} $ as $ \alcovefont{A}^{\prime} $. Since half-spaces are \emph{simplicially-convex}, it therefore suffices to show only Umbrella Property (\ref{UMBextendstominimal}), i.e. that $ ( \alcovefont{B}_0, \ldots, \alcovefont{B}_n ) $ can be extended to a \emph{minimal} gallery connecting $ \alcovefont{B}_0 $ to $ \alcovefont{A}^{\prime} $ (because then both endpoints of the gallery, and therefore the whole gallery, must be contained in that half-space).

Let $ H_i \defeq \alcovefont{B}_i \vert \alcovefont{B}_{i+1} $ ($ i = 0, \ldots, n-1 $) be all the intermediate walls of the gallery $ ( \alcovefont{B}_0, \ldots, \alcovefont{B}_n ) $. Note that $ H_i $ separates $ \alcovefont{B}_i $ from $ \alcovefont{A} $ for all $ 0 \leq i < n-1 $ by hypothesis (\ref{ENUMinductionhypoth}) (more specifically, Umbrella Property (\ref{UMBextendstominimal})) and for $ i = n-1 $ by hypothesis (\ref{ENUMcontradictionhypoth}). By the observation preceding this proof, it therefore suffices to show that the alcoves $ \alcovefont{A} $ and $ \alcovefont{A}^{\prime} $ are on the \emph{same} side of $ H_i $ for all $ 0 \leq i \leq n - 1 $. But this is obviously true: if the claim were \emph{false} for $ H_i $, then necessarily $ H_i = H $, the only hyperplane separating $ \alcovefont{A} $ from $ \alcovefont{A}^{\prime} $, which would mean that $ H $ separated $ \alcovefont{B}_i $ from $ \alcovefont{A} $, a contradiction to hypothesis (\ref{ENUMinductionhypoth}) (more specifically, Umbrella Property (\ref{UMBsamehalfspace})).
\end{proof}

\begin{prop} \label{PantidominantnontranslationshaveCP}
Fix $ w \in \widetilde{W} $.

If $ w \notin \translationsubgroup $ and $ w ( \basevertex ) \in \closedbasechamber^{\opp} $ then $ w $ has the Diamond Property.
\end{prop}

\begin{remark}
View (sequentially!) Figures \ref{example-g2-inductionlemma00} to \ref{example-g2-inductionlemma04} for a picture of the use of Induction Lemma in this proof.
\end{remark}

\begin{proof}
Let $ \alcovefont{B} $ be the $ w $-marked alcove $ w ( \basealcove ) $. By Lemma \ref{Lequivalentdiamonds}, it suffices to show that $ \{ \basealcove, \alcovefont{B} \} $ has the Diamond Property. By Lemma \ref{Ldoublechamberlemma}, either $ \basechamber \cap \weylchambersymbol_{\alcovefont{B}} = \emptyset $ or $ \basechamber \subset \weylchambersymbol_{\alcovefont{B}} $. If $ \basechamber \subset \weylchambersymbol_{\alcovefont{B}} $ then the claim follows from Proposition \ref{PdominantnontranslationshaveDCP}. (using origin $ \vertexsymbol_{\alcovefont{B}} $ and dominant chamber $ \weylchambersymbol_{\alcovefont{B}} $; see Remark \ref{Rindependenceofchamber}). So, assume that $ \basechamber \cap \weylchambersymbol_{\alcovefont{B}} = \emptyset $.

Applying Lemma \ref{Lgallerytoinfinity} to the chambers $ \basechamber $ and $ \weylchambersymbol_{\alcovefont{B}} $ yields a certain infinite minimal gallery $ ( \alcovefont{A}_0 = \basealcove, \alcovefont{A}_1, \alcovefont{A}_2, \ldots ) $ within $ \basechamber $. As usual, give $ \basealcove $ the base marking and let $ ( s_0, s_1, s_2, \ldots ) $ be the infinite sequence in $ \Delta_{\aff} $ such that $ \alcovefont{A}_1 = \alcovefont{A}_0^{s_0} $, $ \alcovefont{A}_2 = \alcovefont{A}_1^{s_1} $, etc. Let $ H_i $ be the wall of the marked alcove $ \alcovefont{A}_i $ labeled by $ s_i $.

Similarly, use the sequence $ ( s_0, s_1, s_2, \ldots ) $ to define, relative to the prescribed labeling of $ \alcovefont{B} $, a corresponding infinite gallery:
\begin{equation*}
( \alcovefont{B}_0, \alcovefont{B}_1, \alcovefont{B}_2, \ldots ) \defeq ( \alcovefont{B}, \alcovefont{B}^{s_0}, ( \alcovefont{B}^{s_0} )^{s_1}, \ldots )
\end{equation*}
As before, each $ \alcovefont{B}_i $ here represents a \emph{marked} alcove. Note that by definition of the labeling $ \labelingsymbol_{\alcovefont{B}} = \labelingsymbol_w $, the gallery $ ( \alcovefont{B}_0, \alcovefont{B}_1, \alcovefont{B}_2, \ldots ) $ is simply the image under $ w $ of the gallery $ ( \alcovefont{A}_0, \alcovefont{A}_1, \alcovefont{A}_2, \ldots ) $. In particular, $ ( \alcovefont{B}_0, \alcovefont{B}_1, \alcovefont{B}_2, \ldots ) $ is \emph{minimal} and $ \alcovefont{B}_i \subset \weylchambersymbol_{ \alcovefont{B} } $ for all $ i $.

Because the gallery $ ( \alcovefont{A}_0, \alcovefont{A}_1, \alcovefont{A}_2, \ldots ) $ starts at $ \basealcove $ and is contained completely within $ \basechamber $, necessarily $ s_0 = s_{\aff} $. Because of this and the hypothesis on $ \vertexsymbol_{ \alcovefont{B} } $, Lemma \ref{Lnegativechambercontainment} says that alcoves $ \alcovefont{B}_0 $ and $ \alcovefont{A}_0 $ are on the \emph{same} side of $ H_0 $, i.e. $ d ( \alcovefont{B}_0, \alcovefont{A}_1 ) = d ( \alcovefont{B}_0, \alcovefont{A}_0 ) + 1 $.

Let $ K \defeq \labelingsymbol_{\alcovefont{B}_0} ( s_0 ) $ be the wall of $ \alcovefont{B}_0 $ labeled by $ s_0 $. If $ \alcovefont{B}_0 $ and $ \alcovefont{A}_1 $ are on the \emph{same} side of $ K $ then necessarily $ K \neq H_0 $ (because $ H_0 $ \emph{separates} $ \alcovefont{B}_0 $ from $ \alcovefont{A}_1 $) and both $ \alcovefont{B}_0 $ and $ \alcovefont{A}_0 $ are on the \emph{same} side of $ K $ (because $ H_0 $ is the unique hyperplane separating $ \alcovefont{A}_0 $ from $ \alcovefont{A}_1 $ and $ H_0 \neq K $). It is then immediate from the definition that $ \{ \alcovefont{A}_0, \alcovefont{B}_0 \} = \{ \basealcove, \alcovefont{B} \} $ has the Direct Diamond Property (realized by $ s_0 $). Otherwise, $ K $ \emph{separates} $ \alcovefont{B}_0 $ from $ \alcovefont{A}_1 $ and by Lemma \ref{Lgeometriclength}, $ d ( \alcovefont{B}_1, \alcovefont{A}_1 ) = d ( \alcovefont{B}_0, \alcovefont{A}_1 ) - 1 = d ( \alcovefont{B}_0, \alcovefont{A}_0 ) + 1 - 1 = d ( \alcovefont{B}_0, \alcovefont{A}_0 ) $, i.e. $ ( \alcovefont{A}_1, \alcovefont{B}_1 ) $ is laterally-conjugate to $ ( \basealcove, \alcovefont{B} ) $ via $ s_0 $.

In these circumstances, Induction Lemma implies that $ (  ( \alcovefont{B}_0, \alcovefont{B}_1 ), \alcovefont{A}_1, H_1 ) $ is an Umbrella:
\begin{itemize}
\item the non-numbered hypotheses of Induction Lemma are true by choice,

\item hypothesis (\ref{ENUMinductionhypoth}) is true because $ ( ( \alcovefont{B}_0 ), \alcovefont{A}_0, H_0 ) $ is trivially an Umbrella,

\item hypothesis (\ref{ENUMvertexinchamber}) is true by hypothesis on $ w $,

\item hypothesis (\ref{ENUMawayfrombase}) is true by choice of $ H_1 $ because $ ( \alcovefont{A}_0, \alcovefont{A}_1, \alcovefont{A}_2, \ldots ) $ is \emph{minimal},

\item hypothesis (\ref{ENUMnotbasewall}) is true by choice of $ H_1 $ because $ \alcovefont{A}_1, \alcovefont{A}_2 \subset \basechamber $, and

\item hypothesis (\ref{ENUMcontradictionhypoth}) is true by the assumption that $ \{ \alcovefont{A}_0, \alcovefont{B}_0 \} $ did \emph{not} have the Direct Diamond Property for $ s_0 $ (see previous paragraph: by choice $ K = \alcovefont{B}_0 \vert \alcovefont{B}_1 $).
\end{itemize}

So, the triple $ ( ( \alcovefont{B}_0, \alcovefont{B}_1 ), \alcovefont{A}_1, H_1 ) $ is an Umbrella by Induction Lemma. In particular, $ d (  \alcovefont{B}_1, \alcovefont{A}_2 ) = d ( \alcovefont{B}_1, \alcovefont{A}_1 ) + 1 $ by Umbrella Property (\ref{UMBsamehalfspace}).

I now iterate this process.

Let $ K \defeq \labelingsymbol_{\alcovefont{B}_1} ( s_1 ) $ be the wall of the marked alcove $ \alcovefont{B}_1 $ labeled by $ s_1 $. If $ \alcovefont{B}_1 $ and $ \alcovefont{A}_2 $ are on the \emph{same} side of $ K $ then $ K \neq H_1 $ (because $ H_1 $ \emph{separates} $ \alcovefont{B}_1 $ from $ \alcovefont{A}_2 $) and both $ \alcovefont{B}_1 $ and $ \alcovefont{A}_1 $ are on the \emph{same} side of $ K $ (because $ H_1 $ is the unique hyperplane separating $ \alcovefont{A}_1 $ from $ \alcovefont{A}_2 $ and $ H_1 \neq K $). It is then immediate from the definition that $ \{ \alcovefont{A}_1, \alcovefont{B}_1 \} $ has the Direct Diamond Property realized by $ s_1 $ and therefore $ \{ \basealcove, \alcovefont{B} \} $, being laterally-conjugate to it, has the Diamond Property. Otherwise, $ K $ separates $ \alcovefont{B}_1 $ from $ \alcovefont{A}_2 $ and by Lemma \ref{Lgeometriclength}, $ d (  \alcovefont{B}_2, \alcovefont{A}_2 ) = d (  \alcovefont{B}_1, \alcovefont{A}_2 ) - 1 = d ( \alcovefont{B}_1, \alcovefont{A}_1 ) + 1 - 1 = d ( \alcovefont{B}_1, \alcovefont{A}_1 ) $, i.e. $ ( \alcovefont{A}_2, \alcovefont{B}_2 ) $ is laterally-conjugate to $ ( \alcovefont{A}_1, \alcovefont{B}_1 ) $ via $ s_1 $, and therefore also laterally-conjugate to $ ( \basealcove, \alcovefont{B} ) $.

In these circumstances, Induction Lemma implies that $ ( ( \alcovefont{B}_0, \alcovefont{B}_1, \alcovefont{B}_2 ), \alcovefont{A}_2, H_2 ) $ is an Umbrella:
\begin{itemize}
\item the non-numbered hypotheses are again true by choice, the status of hypothesis (\ref{ENUMvertexinchamber}) has not changed, and hypothesis (\ref{ENUMawayfrombase}) is true by choice of $ H_2 $ for the same reason as before,

\item hypothesis (\ref{ENUMinductionhypoth}) is known by the previous iteration,

\item hypothesis (\ref{ENUMnotbasewall}) is true by choice of $ H_2 $ because $ \alcovefont{A}_2, \alcovefont{A}_3 \subset \basechamber $, and

\item hypothesis (\ref{ENUMcontradictionhypoth}) is supplied by the assumption that $ \{ \alcovefont{A}_1, \alcovefont{B}_1 \} $ did \emph{not} have the Direct Diamond Property for $ s_1 $ (see previous paragraph: by choice $ K = \alcovefont{B}_1 \vert \alcovefont{B}_2 $).
\end{itemize}

The above induction shows that if $ n \in \N $ and $ \{ \alcovefont{A}_i, \alcovefont{B}_i \} $ does not have the Direct Diamond Property for all $ i \leq n $ then $ ( \alcovefont{A}_i, \alcovefont{B}_i ) $ is laterally-conjugate to $ ( \basealcove, \alcovefont{B} ) $ for all $ i \leq n $, and in particular, $ d ( \alcovefont{A}_i, \alcovefont{B}_i ) = ( \basealcove, \alcovefont{B} ) = \ell ( w ) $ for all $ i \leq n $. But Lemma \ref{Lgallerytoinfinity} says that $ d ( \alcovefont{A}_i, \alcovefont{B}_i ) \rightarrow \infty $, so there must exist $ i \in \N $ such that $ \{ \alcovefont{A}_i, \alcovefont{B}_i \} $ has the Direct Diamond Property. If $ i \in \N $ is the smallest such index then $ ( \basealcove, \alcovefont{B} ) $ is laterally-conjugate to a pair with the Direct Diamond Property, as desired.
\end{proof}

\begin{remark} \label{Rindependenceofchamber2}
Similar to the ``dominant case'' Proposition \ref{PdominantnontranslationshaveDCP}, the choices of vertex, alcove, and chamber are notationally convenient but otherwise totally unimportant to the conclusion of Proposition \ref{PantidominantnontranslationshaveCP}. In other words, if $ \alcovefont{A} $ and $ \alcovefont{B} $ are NT-compatible marked alcoves and $ \vertexsymbol_{\alcovefont{B}} \in \closedweylchambersymbol_{\alcovefont{A}}^{\opp} $ then $ \{ \alcovefont{A}, \alcovefont{B} \} $ has the Diamond Property.
\end{remark}

\begin{remark}
In type A, there is very simple proof that if $ w ( \basealcove ) \subset \basechamber^{\opp} $ then $ w $ has the Diamond Property using at most one lateral-conjugation. Because of the extreme symmetry of type A, the inverse image $ s_{\aff}^{-1} ( \basechamber ) $ consists of $ \basechamber^{\opp} $ together with \emph{all} alcoves $ \alcovefont{B} $ such that $ \overline{\alcovefont{B}} \cap \overline{\weylchambersymbol}^{\opp} \neq \emptyset $. Because of this, if $ w $ does not already have the Direct Diamond Property, $ s_{\aff} $ laterally conjugates $ w $ into the dominant chamber, in which case Proposition \ref{PdominantnontranslationshaveDCP} applies. Of course, this idea fails in (almost?) every other type.
\end{remark}

\subsection{Case: the intermediate chambers} \label{SSintermediatecase}

I now prove that the general case can, at worst, be reduced to the anti-dominant case, Proposition \ref{PantidominantnontranslationshaveCP}:
\begin{maintheorem}
Suppose $ \alcovefont{A} $ and $ \alcovefont{B} $ are marked alcoves.

If $ \alcovefont{A} $ and $ \alcovefont{B} $ are NT-compatible then $ \{ \alcovefont{A}, \alcovefont{B} \} $ has the Diamond Property.

In particular, if $ w \in \widetilde{W} $ and $ w \notin \translationsubgroup $ then $ w $ has the Diamond Property.
\end{maintheorem}

\begin{remark}
View (sequentially!) Figures \ref{example-g2-2ndinductionlemma-00} to \ref{example-g2-2ndinductionlemma-01} for a picture of the iteration used in this proof.
\end{remark}

\begin{proof}
Let $ \mathcal{S} $ be the set of all $ s \in \Delta_{\aff} $ such that if $ H \defeq \labelingsymbol_{\alcovefont{A}} ( s ) $ then the following three properties are true simultaneously: $ \vertexsymbol_{\alcovefont{A}} \in H $, both $ \alcovefont{A} $ and $ \alcovefont{B} $ are on the \emph{same} side of $ H $, and $ \vertexsymbol_{\alcovefont{B}} \notin H $. If $ \mathcal{S} = \emptyset $ then by definition for every wall $ H $ of $ \alcovefont{A} $ containing $ \vertexsymbol_{\alcovefont{A}} $ either $ H $ \emph{separates} $ \alcovefont{A} $ from $ \alcovefont{B} $ or $ \vertexsymbol_{\alcovefont{B}} \in H $. In this case, $ \vertexsymbol_{\alcovefont{B}} \in \overline{\weylchambersymbol}_{\alcovefont{A}}^{\opp} $ and Proposition \ref{PantidominantnontranslationshaveCP} applies. So assume that $ \mathcal{S} \neq \emptyset $ and let $ s \in \mathcal{S} $ be arbitrary.

Let $ ( \alcovefont{B}_0, \ldots, \alcovefont{B}_d ) $ be a gallery realizing the distance $ d ( \vertexsymbol_{\alcovefont{B}}, \weylchambersymbol_{\alcovefont{A}}^{\opp} ) $ (see \S\ref{SSgalleries} for this notion of distance). By definition, this means that $ \vertexsymbol_{\alcovefont{B}} \in \overline{\alcovefont{B}}_0 $, the gallery is minimal, and $ \alcovefont{B}_d \subset \weylchambersymbol_{\alcovefont{A}}^{\opp} $ (note that $ \alcovefont{B}_0 \neq \alcovefont{B} $ is possible). Set $ H \defeq \labelingsymbol_{\alcovefont{A}} ( s ) $. By definition of $ \mathcal{S} $, $ H $ separates $ \alcovefont{B} $ from every alcove in $ \weylchambersymbol_{\alcovefont{A}}^{\opp} $ and $ \vertexsymbol_{\alcovefont{B}} \notin H $ so $ H $ also separates $ \alcovefont{B}_0 $ from every alcove in $ \weylchambersymbol_{\alcovefont{A}}^{\opp} $. Therefore, $ H $ must be an intermediate wall of the gallery $ ( \alcovefont{B}_0, \ldots, \alcovefont{B}_d ) $ (see Lemma 5.1.5 of \cite{BGW}). Let $ 0 \leq j < d $ be the index such that $ H = \alcovefont{B}_j \vert \alcovefont{B}_{j+1} $. Then the sequence of alcoves $ ( \alcovefont{B}_0, \ldots, \alcovefont{B}_j, s_H ( \alcovefont{B}_{j+2} ), \ldots, s_H ( \alcovefont{B}_d ) ) $ is a \emph{gallery}. This is obviously a gallery ``from $ \vertexsymbol_{\alcovefont{B}} $ to $ \weylchambersymbol_{\alcovefont{A}^s}^{\opp} $'' and has fewer than $ d $ alcoves. By choice of $ s $ and the compatibility hypothesis, $ \vertexsymbol_{\alcovefont{B}^s} = \vertexsymbol_{\alcovefont{B}} $. Altogether, $ d ( \vertexsymbol_{\alcovefont{B}^s}, \weylchambersymbol_{\alcovefont{A}^s}^{\opp} ) < d ( \vertexsymbol_{\alcovefont{B}}, \weylchambersymbol_{\alcovefont{A}}^{\opp} ) $.

On the other hand, by choice of $ s $, Lemma \ref{Lgeometriclength} implies that $ d ( \alcovefont{A}^s, \alcovefont{B}^s ) \geq d ( \alcovefont{A}, \alcovefont{B} ) $.
Since these distances $ d ( \vertexsymbol_{*}, \weylchambersymbol_{*}^{\opp} ) $ are $ \N $-valued, this means that one may iterate the previous process until a pair of alcoves $ ( \alcovefont{A}^{\prime}, \alcovefont{B}^{\prime} ) $ is constructed which is laterally-conjugate to $ ( \alcovefont{A}, \alcovefont{B} ) $ and such that either $ ( \alcovefont{A}^{\prime}, \alcovefont{B}^{\prime} ) $ has the Direct Diamond Property or $ \vertexsymbol_{\alcovefont{B}^{\prime}} \in \overline{\weylchambersymbol}_{\alcovefont{A}^{\prime}}^{\opp} $, in which case Proposition \ref{PantidominantnontranslationshaveCP} applies.
\end{proof}

\begin{remark}
If $ w \in \finiteweylgroup \subset \widetilde{W} $ then $ w \notin \Lambda $ if and only if $ w \neq 1 $, and it is easy to show that both $ \ell ( w s_{\aff} ) > \ell ( w ) $ and $ \ell ( s_{\aff} w ) > \ell ( w ) $ directly: combine the Exchange Property of the Coxeter group $ ( W_{\aff}, \Delta_{\aff} ) $ with the fact that all reduced expressions for a single element must use the \emph{same} subset of $ \Delta_{\aff} $ (see Proposition 7 in Ch IV \S1 no. 8 of \cite{bourbaki}) to conclude that neither length can \emph{decrease}. It is tempting to think that such $ w $ always have the Direct Diamond Property realized by $ s_{\aff} $ but this is not always true: in the affine Weyl group $ \Caff{2} $ (or $ \Gaff{2} $), there exists $ s \in \basereflections $ such that $ s \cdot s_{\aff} = s_{\aff} \cdot s $, so $ s_{\aff} w s_{\aff} = w $ for $ w \defeq s \in \finiteweylgroup $.
\end{remark}

\section{Application to Hecke Algebras} \label{Sapplicationtoheckealgebras}

\subsection{Hecke algebras on quasi-Coxeter groups} \label{SSdefofheckealgebra}

Fix a function $ q : \Delta_{\aff} \rightarrow \N $ which is invariant under conjugation by $ \widetilde{W} $.

In the rest of this section \S\ref{Sapplicationtoheckealgebras}, I assume given a $ \C $-algebra $ \heckefont{H} $ which, as a $ \C $-vector space, has a basis of elements $ \heckefont{T}_w $ indexed by all $ w \in \widetilde{W} $. Further, denoting the ring operation by $ * $, I assume that the following \emph{Iwahori-Matsumoto identities} are true in $ \heckefont{H} $: for all $ w \in \widetilde{W} $ and $ s \in \Delta_{\aff} $,
\begin{align*}
\heckefont{T}_s * \heckefont{T}_w &=
\begin{cases}
\heckefont{T}_{sw} &\text{ if } \ell ( s w ) > \ell ( w ) \\
( q(s) - 1 ) \heckefont{T}_w + q(s) \heckefont{T}_{sw} &\text{ if } \ell ( s w ) < \ell ( w )
\end{cases} \text{ (left-handed)} \\
\heckefont{T}_w * \heckefont{T}_s &=
\begin{cases}
\heckefont{T}_{ws} &\text{ if } \ell ( w s ) > \ell ( w ) \\
( q(s) - 1 ) \heckefont{T}_w + q(s) \heckefont{T}_{ws} &\text{ if } \ell ( w s ) < \ell ( w )
\end{cases} \text{ (right-handed)}
\end{align*}
Note that because of the way that the length function $ \ell $ was extended to $ \widetilde{W} $, if $ \tau \in \Omega $ and $ w \in \widetilde{W} $ then $ \heckefont{T}_{w \tau} = \heckefont{T}_w * \heckefont{T}_{\tau} $. If $ h \in \heckefont{H} $ then denote by $ h_w $ the coefficient of $ \heckefont{T}_w $ in the linear combination of $ h $ with respect to this basis. If $ h \in \heckefont{H} $ and $ h_w \neq 0 $ then $ w $ is said to \emph{support} $ h $.

\begin{remark}
It is not difficult to show, and I do so in the article \cite{roro} using ingredients from the Appendix to \cite{PRH}, that any Iwahori-Hecke algebra $ \heckefont{H} $ of any connected reductive affine algebraic $ F $-group is of the form described above. Therefore, the results of this section \S\ref{Sapplicationtoheckealgebras} apply to Iwahori-Hecke algebras. If greater generality is desired, one can use a pair $ a, b : \Delta_{\aff} \rightarrow \C $ of parameter systems and a ``generic algebra'' as in \S7.1 of \cite{humphreys}.
\end{remark}

I will need the following slight extension of a well-known property of Coxeter groups:
\begin{lemma} \label{Lhumphreys}
Fix $ w \in \widetilde{W} $ and $ s, t \in \Delta_{\aff} $.

If $ \ell ( s w t ) = \ell ( w ) $ and $ \ell ( s w ) = \ell ( w t ) $ then $ s w t = w $.
\end{lemma}

\begin{proof}
When $ \Omega = \{ 1 \} $, this is exactly Lemma in \S7.2 of \cite{humphreys}. The general case follows immediately from this since $ \Omega $ permutes $ \Delta_{\aff} $ and $ \ell $ factors through $ W_{\aff} $.
\end{proof}

\subsection{Equations defining the center} \label{SScentralityequations}

Denote by $ Z ( \heckefont{H} ) $ the center of the ring $ \heckefont{H} $.

Fix $ h \in \heckefont{H} $. It is clear from the Iwahori-Matsumoto relations that $ h \in Z ( \heckefont{H} ) $ if and only if $ h * \heckefont{T}_s = \heckefont{T}_s * h $ and $ h * \heckefont{T}_{\tau} = \heckefont{T}_{\tau} * h $ for all $ s \in \Delta_{\aff} $ and $ \tau \in \Omega $.

Fix $ s \in \Delta_{\aff} $. For each $ x \in \widetilde{W} $, one can use the left-handed Iwahori-Matsumoto relation to compute that the coefficient of $ \heckefont{T}_x $ in $ \heckefont{T}_s * h $ is
\begin{align*}
q(s) h_{ s x } &\text{ if } \ell ( s x ) > \ell ( x ) \\
h_{ s x } + ( q(s) - 1 ) h_x &\text{ if } \ell ( s x ) < \ell ( x )
\end{align*}

Similarly, one can use the right-handed Iwahori-Matsumoto relation to compute that the coefficient of $ \heckefont{T}_x $ in $ h * \heckefont{T}_s $ is
\begin{align*}
q(s) h_{ x s } &\text{ if } \ell ( x s ) > \ell ( x ) \\
h_{ x s } + ( q(s) - 1 ) h_x &\text{ if } \ell ( x s ) < \ell ( x )
\end{align*}

It is obvious from the Iwahori-Matsumoto identities that $ h * \heckefont{T}_{\tau} = \heckefont{T}_{\tau} * h $ if and only if $ h_{x \tau} = h_{\tau x} $ for all $ x \in \widetilde{W} $.

It follows that the center $ Z ( \heckefont{H} ) $ is the $ \C $-subspace of vectors $ h \in \heckefont{H} $ whose Iwahori-Matsumoto coefficients $ h_x $ solve the (infinite) linear system consisting of the equation $ h_{x \tau} = h_{\tau x} $ for each pair $ ( x, \tau ) \in \widetilde{W} \times \Omega $ together with the appropriate equation from
\begin{align}
\nonumber q(s) h_{ s x } = q(s) h_{ x s } &\text{ if } \ell ( s x ), \ell ( x s ) > \ell ( x ) \\
\label{Ecenterequations_us_lowest} q(s) h_{ s x } = h_{ x s } + ( q(s) - 1 ) h_x &\text{ if } \ell ( s x ) > \ell ( x ) > \ell ( x s ) \\
\nonumber h_{ s x } + ( q(s) - 1 ) h_x = q(s) h_{ x s } &\text{ if } \ell ( s x ) < \ell ( x ) < \ell ( x s ) \\
\label{Ecenterequations_bothdecrease} h_{ s x } = h_{ x s } &\text{ if } \ell ( s x ), \ell ( x s ) < \ell ( x )
\end{align}
for each pair $ ( x, s ) \in \widetilde{W} \times \Delta_{\aff} $.

\begin{remark}
These equations appeared already in \S3 of \cite{haines} for affine Hecke algebras on reduced root data.
\end{remark}

\subsection{Length-filtration and dimensions} \label{SSdimensionbounds}

Recall that $ \Omega ( w ) $ denotes the projection of $ w \in \widetilde{W} $ into $ \Omega $, and that the set of $ \finiteweylgroup $-conjugacy classes in $ \translationsubgroup $ is denoted by $ \translationsubgroup / \finiteweylgroup $.

\begin{defn}
Fix $ L \in \N $ and $ \tau \in \Omega $.

Define $ Z_{L,\tau} ( \heckefont{H} ) $ to be the set of all $ z \in Z ( \heckefont{H} ) $ such that $ z_w = 0 $ if either $ \ell ( w ) > L $ or $ \Omega ( w ) \neq \tau $.
\end{defn}
Note that each $ Z_{L,\tau} ( \heckefont{H} ) $ is a \emph{finite}-dimensional $ \C $-subspace of $ Z ( \heckefont{H} ) $ and that $ Z ( \heckefont{H} ) $ is the \emph{union} of all $ Z_{L,\tau} ( \heckefont{H} ) $.

Recall that if $ \mathcal{O} \in \translationsubgroup / \finiteweylgroup $ then $ \ell $ is constant on $ \mathcal{O} $ and define $ \ell ( \mathcal{O} ) $ to be this constant length. It follows that any two $ t, t^{\prime} \in \mathcal{O} $ are \emph{laterally-conjugate}.

\begin{defn}
Fix $ L \in \N $ and $ \tau \in \Omega $.

Define $ N_{L, \tau} $ to be the total number of conjugacy classes $ \mathcal{O} \in \translationsubgroup / \finiteweylgroup $ such that $ \ell ( \mathcal{O} ) \leq L $ and $ \Omega ( t ) = \tau $ for all $ t \in \mathcal{O} $.
\end{defn}

The following two lemmas show how lateral-conjugacy and the diamond property are related to centers of Hecke algebras:

\begin{lemma} \label{Llateralconjugacyimpliesdependence}
Suppose $ z \in Z ( \heckefont{H} ) $.

If $ w \in \widetilde{W} $ is laterally-conjugate to $ w^{\prime} $ then $ z_w = z_{w^{\prime}} $.
\end{lemma}

\begin{proof}
If $ s \in \Delta_{\aff} $ is such that $ \ell ( s w s ) = \ell ( w ) $ then either $ \ell ( s w ) < \ell ( w ) < \ell ( w s ) $ or $ \ell ( w s ) < \ell ( w ) < \ell ( s w ) $. Choosing $ x \defeq w s $ in the former case and $ x \defeq s w $ in the latter case, centrality equation (\ref{Ecenterequations_bothdecrease}) implies that $ z_{w} = z_{s w s} $, and the claim follows immediately from this.
\end{proof}

\begin{lemma} \label{Ldiamondpropertyimpliesdependence}
Suppose $ z \in Z ( \heckefont{H} ) $.

If $ w \in \widetilde{W} $ has the Diamond Property then there exist $ u, v \in \widetilde{W} $ satisfying $ \ell ( u ) > \ell ( v ) > \ell ( w ) $ such that $ z_w $ is a $ \C $-linear combination of $ z_u $ and $ z_v $.
\end{lemma}

\begin{proof}
By Lemma \ref{Llateralconjugacyimpliesdependence}, I may assume that $ w $ has the \emph{Direct} Diamond Property. Let $ s \in \Delta_{\aff} $ be the element realizing the property. By basic Coxeter theory, $ \ell ( s w ) = \ell ( w ) + 1 = \ell ( w s ) $ and it is true that either $ \ell ( s w s ) = \ell ( w ) + 2 $ or $ \ell ( s w s ) = \ell ( w ) $. If it were true that $ \ell ( s w s ) = \ell ( w ) $ then by Lemma \ref{Lhumphreys} it would be true that $ s w s = w $, but this is explicitly prohibited by the Direct Diamond Property. Therefore, $ \ell ( s w s ) > \ell ( s w ) = \ell ( w s ) > \ell ( w ) $. Choosing $ x \defeq w s $ and applying centrality equation (\ref{Ecenterequations_us_lowest}) proves that $ z_w $ is a linear combination of $ z_{ws} $ and $ z_{sws} $, as desired.
\end{proof}

\begin{remark}
Materially, both lemmas \ref{Llateralconjugacyimpliesdependence} and \ref{Ldiamondpropertyimpliesdependence} appeared already as Lemma 3.1 of \cite{haines}.
\end{remark}

\begin{prop} \label{Pdimensionbounds}
Fix $ L \in \N $ and $ \tau \in \Omega $.
\begin{equation*}
\dim_{\C} ( Z_{L,\tau} ( \heckefont{H} ) ) \leq N_{L,\tau}
\end{equation*}
\end{prop}

\begin{proof}
Suppose $ w \in \widetilde{W} $ and $ w \notin \translationsubgroup $. Consider the linear system defining $ Z ( \heckefont{H} ) $ as a $ \C $-subspace of $ \heckefont{H} $. By applying the Main Theorem and Lemma \ref{Ldiamondpropertyimpliesdependence} repeatedly, one can express the variable $ h_w $ as a $ \C $-linear combination of variables $ h_x $ such that either $ \ell ( x ) > L $ or $ x \in \translationsubgroup $. Since $ Z_{L,\tau} ( \heckefont{H} ) $ is the $ \C $-subspace of $ Z ( \heckefont{H} ) $ defined by the additional equations $ h_x = 0 $ for all $ x \in \widetilde{W} $ such that either $ \ell ( x ) > L $ or $ \Omega ( x ) \neq \tau $, it follows that $ \dim_{\C} ( Z_{L,\tau} ( \heckefont{H} ) ) $ is \emph{at most} the total number of $ t \in \translationsubgroup $ such that $ \ell ( t ) \leq L $ and $ \Omega ( t ) = \tau $. On the other hand, if $ \mathcal{O} \in \translationsubgroup / \finiteweylgroup $ and $ t, t^{\prime} \in \mathcal{O} $ then $ t $ is laterally-conjugate to $ t^{\prime} $ and $ \Omega ( t ) = \Omega ( t^{\prime} ) $, so the dimension bound now follows from Lemma \ref{Llateralconjugacyimpliesdependence}.
\end{proof}

Note one extra detail from the proof: if $ z \in Z_{L,\tau} ( \heckefont{H} ) $ and both $ w \notin \translationsubgroup $ and $ \ell ( w ) = L $ then $ z_w = 0 $.

\begin{corollary}
Suppose that for each conjugacy class $ \mathcal{O} \in \translationsubgroup / \finiteweylgroup $ there is an element $ z_{ \mathcal{O} } \in Z ( \heckefont{H} ) $ such that $ \ell ( w ) \leq \ell ( \mathcal{O} ) $ for all $ w \in \widetilde{W} $ supporting $ z_{ \mathcal{O} } $ and such that $ \Omega ( w ) $ is the same for all $ w \in \widetilde{W} $ supporting $ z_{ \mathcal{O} } $.

If $ \{ z_{ \mathcal{O} } \}_{ \mathcal{O} \in \translationsubgroup / \finiteweylgroup } $ is linearly-independent then it is a basis for $ Z ( \heckefont{H} ) $.
\end{corollary}

\begin{proof}
Fix $ L \in \N $ and $ \tau \in \Omega $. Consider only those $ z_{ \mathcal{O} } $ for which $ \ell ( \mathcal{O} ) \leq L $ and for which the uniform $ \Omega $-component of those $ w $ supporting $ z_{ \mathcal{O} } $ is $ \tau $. By hypothesis, the set of all such $ z_{ \mathcal{O} } $ is a linearly-independent subset of $ N_{L,\tau} $ vectors in the subspace $ Z_{L,\tau} ( \heckefont{H} ) $. By Proposition \ref{Pdimensionbounds}, this set must also \emph{span} $ Z_{L,\tau} ( \heckefont{H} ) $. Since $ Z ( \heckefont{H} ) $ is the union over all pairs $ ( L, \tau ) $ of the subspaces $ Z_{L,\tau} ( \heckefont{H} ) $, the claim follows.
\end{proof}

\newpage

\section{Pictures} \label{Sfigures}

\begin{figure}[h]
\centering
\includegraphics[]{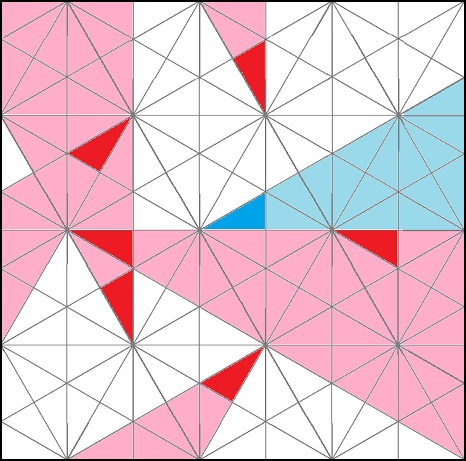}
\caption{The Main Theorem does not follow directly from Proposition \ref{PdominantnontranslationshaveDCP} (the dominant case). The blue alcove in the center is the base alcove $ \basealcove $ and the red alcoves surrounding it constitute a full lateral-conjugacy class. The light blue/pink cones are the unique dominant Weyl chamber containing the various alcoves (uniqueness is due to the fact that in this $ \Gaff{2} $ example each alcove contains only one special vertex in its closure).}
\label{example-g2-annoyingorbit}
\end{figure}

\newpage

\begin{figure}[h]
\centering
\includegraphics[]{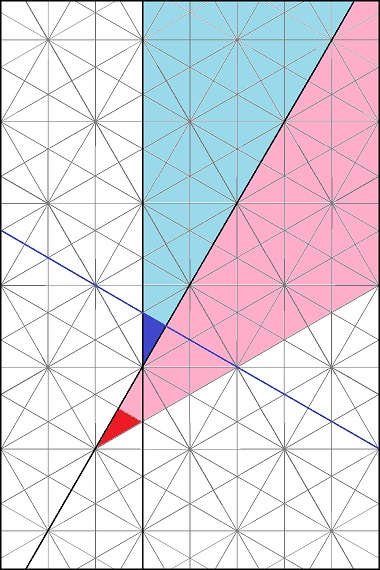}
\caption{Sample initial situation in the proof of Proposition \ref{PantidominantnontranslationshaveCP}. The purple alcove is $ \basealcove $ and the red alcove is $ \alcovefont{B} $. The light blue cone is the dominant Weyl chamber $ \basechamber $ and the pink cone is the chamber $ \weylchambersymbol_{\alcovefont{B}} $.}
\label{example-g2-inductionlemma00}
\end{figure}

\newpage

\begin{figure}[h]
\centering
\includegraphics[]{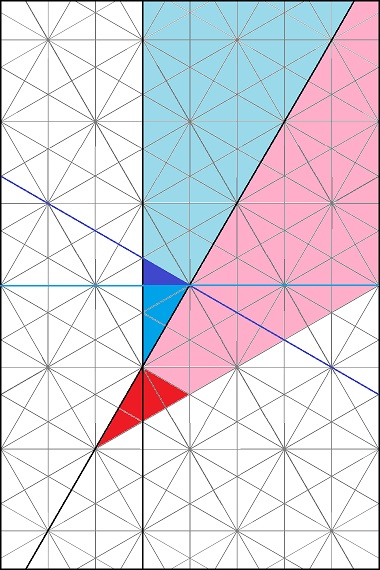}
\caption{After 3 iterations of Induction Lemma. The purple alcove is $ \alcovefont{A}^{\prime} $ and all possibilities for $ H^{\prime} $ (only one in this case) are also purple. All other alcoves in the gallery $ \alcovefont{A}_{\bullet} $ are blue, in particular the alcove $ \alcovefont{A} $ adjacent to $ \alcovefont{A}^{\prime} $. The blue wall is $ H $. The red alcoves constitute the gallery $ \alcovefont{B}_{\bullet} $. Observe that the current $ H $ is always (one of) the previous $ H^{\prime} $.}
\label{example-g2-inductionlemma03}
\end{figure}

\newpage

\begin{figure}[h]
\centering
\includegraphics[]{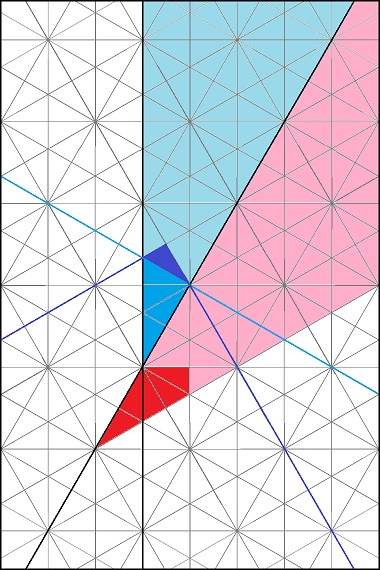}
\caption{After 4 iterations of Induction Lemma. Observe that (for the first time) the purple alcove and the nearest red alcove (its lateral conjugate) have the Direct Diamond Property. Nonetheless, the conclusion of Induction Lemma remains true for several more iterations. An important observation is that the conclusion of Induction Lemma \textbf{must} eventually fail \textbf{because} the Diamond Property is true.}
\label{example-g2-inductionlemma04}
\end{figure}

\newpage

\begin{figure}[h]
\centering
\includegraphics[]{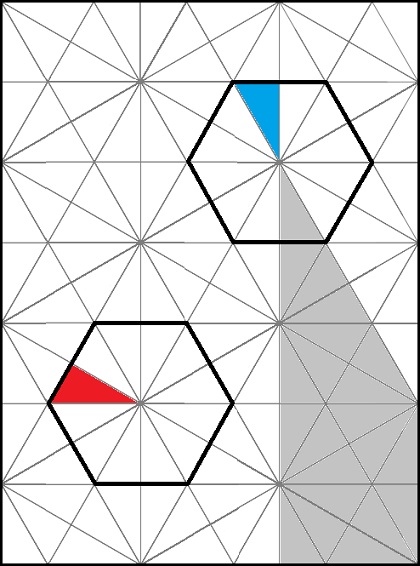}
\caption{Situation in which $ \mathcal{S} \neq \emptyset $ and $ d ( \vertexsymbol_{ \alcovefont{B} }, \weylchambersymbol_{\alcovefont{A}}^{\opp} ) > 0 $ in the proof of the Main Theorem (although in this particular example, the pair of alcoves already has the Direct Diamond Property realized by $ s_{\aff} $ so no action is necessary). The blue alcove is $ \alcovefont{A} $, the red alcove is $ \alcovefont{B} $, and the grey cone is $ \weylchambersymbol_{\alcovefont{A}}^{\opp} $. The black outline is merely a visual aid.}
\label{example-g2-2ndinductionlemma-00}
\end{figure}

\newpage

\begin{figure}[h]
\centering
\includegraphics[]{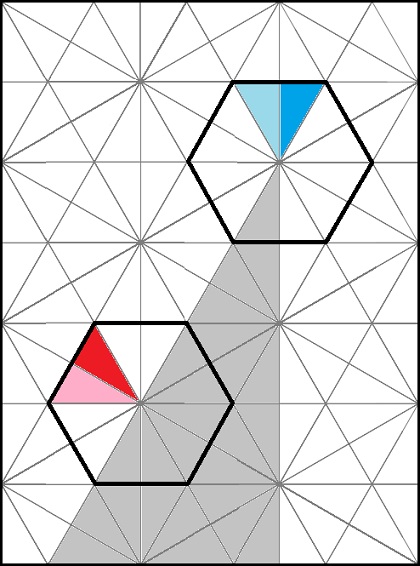}
\caption{After 1 iteration of the process described in the proof of the Main Theorem. The value $ d ( \vertexsymbol_{ \alcovefont{B} }, \weylchambersymbol_{\alcovefont{A}}^{\opp} ) $ is now $ 0 $ and $ \vertexsymbol_{ \alcovefont{B} } \in \overline{ \weylchambersymbol }_{\alcovefont{A}}^{\opp} $. In this particular example, it is possible to laterally conjugate once more to arrange $ \alcovefont{B} \subset \weylchambersymbol_{\alcovefont{A}}^{\opp} $, but this is not always possible. The light blue/pink alcove merely represents the previous position of the blue/red alcove.}
\label{example-g2-2ndinductionlemma-01}
\end{figure}

\end{document}